\documentclass[12pt]{article}

\usepackage[]{graphicx}

\usepackage[usenames, dvipsnames]{xcolor}

\usepackage{alltt}

\usepackage{multirow}
\usepackage{hhline}
\usepackage{amsmath}
\usepackage{amsthm}
\usepackage{amssymb} 
\usepackage{amsfonts}
\usepackage{paralist}
\usepackage{caption,subcaption}
\usepackage[]{graphicx}\usepackage[]{color}
\usepackage{geometry}
\usepackage[colorlinks=true, linkcolor=blue, citecolor=gray, linktocpage=true]{hyperref}

\usepackage{algorithm}
\usepackage{algpseudocode}
\usepackage{algorithmicx}

\usepackage{import}

\usepackage[normalem]{ulem}

\makeatletter
\renewcommand{\fnum@figure}{Fig. \thefigure}
\makeatother

\newcommand{\veca}{\mathbf{a}}
\newcommand{\f}{\mathbf{f}}
\newcommand{\x}{\mathbf{x}}
\newcommand{\y}{\mathbf{y}}
\newcommand{\z}{\mathbf{z}}

\newcommand{\vecb}{\mathbf{b}}
\newcommand{\vecc}{\mathbf{c}}
\newcommand{\vecx}{\mathbf{x}}

\newcommand{\vecu}{\mathbf{u}}
\newcommand{\vecv}{\mathbf{v}}

\newcommand{\X}{\mathcal{X}}
\newcommand{\Y}{\mathcal{Y}}
\newcommand{\Z}{\mathcal{Z}}
\newcommand{\A}{\mathbf{A}}
\newcommand{\R}{\mathbb{R}}
\newcommand{\Id}{\mathbf{I}}

\newcommand{\N}{\mathbb{N}}

\newcommand{\U}{\mathcal{U}}
\newcommand{\I}{\mathcal{I}}
\newcommand{\Paral}{\mathcal{P}}
\newcommand{\Band}{\mathcal{S}}

\newcommand{\boxx}{\mathcal{B}}
\newcommand{\B}{\mathbf{B}}
\newcommand{\classA}{\mathcal{A}}

\newcommand{\pbQ}{\mathcal{Q}}
\newcommand{\pbR}{\mathcal{R}}
\newcommand{\Emb}{\mathcal{E}}

\newcommand{\set}[1]{\left\lbrace{#1}\right\rbrace}

\newcommand{\telque}{\,,\,} 

\def\restriction#1#2{\mathchoice
              {\setbox1\hbox{${\displaystyle #1}_{\scriptstyle #2}$}
              \restrictionaux{#1}{#2}}
              {\setbox1\hbox{${\textstyle #1}_{\scriptstyle #2}$}
              \restrictionaux{#1}{#2}}
              {\setbox1\hbox{${\scriptstyle #1}_{\scriptscriptstyle #2}$}
              \restrictionaux{#1}{#2}}
              {\setbox1\hbox{${\scriptscriptstyle #1}_{\scriptscriptstyle #2}$}
              \restrictionaux{#1}{#2}}}
\def\restrictionaux#1#2{{#1\,\smash{\vrule height .8\ht1 depth .85\dp1}}_{\,#2}}

\DeclareMathOperator*{\argmin}{argmin}
\DeclareMathOperator*{\argmax}{argmax}
\DeclareMathOperator*{\Ran}{Ran}
\DeclareMathOperator*{\Vol}{Vol}

\newtheorem{Property}{Property}[section]
\newtheorem{Definition}{Definition}[section]

\newtheorem{Remark}{Remark}[section]
\newtheorem{Proposition}{Proposition}[section]
\newtheorem{theorem}{Theorem}

\title{On the choice of the low-dimensional domain for global optimization via random embeddings}

\author{Micka\"{e}l Binois\thanks{Corresponding author: The University of
Chicago Booth School of Business, 5807 S.~Woodlawn Ave., Chicago, IL, USA;
\href{mailto:mbinois@chicagobooth.edu}{\tt mbinois@chicagobooth.edu}}
\and David Ginsbourger
\thanks{Uncertainty Quantification and Optimal Design group, Idiap Research Institute, Centre
du Parc, Rue Marconi 19, PO Box 592, 1920 Martigny, Switzerland}
\thanks{IMSV, Department of Mathematics and Statistics,  University of Bern,  Alpeneggstrasse
22, 3012 Bern, Switzerland}
\and Olivier Roustant\thanks{EMSE Ecole des Mines de St-\'Etienne, UMR CNRS 6158, LIMOS, F-42023: 158 Cours
Fauriel, Saint-\'Etienne, France}
}

\date{\today}

\IfFileExists{upquote.sty}{\usepackage{upquote}}{}
\begin{document}
\maketitle

\begin{abstract}
The challenge of taking many variables into account in optimization problems may be overcome under the hypothesis of low effective dimensionality.
Then, the search of solutions can be reduced to the random embedding of a low dimensional space into the original one, resulting in a more manageable optimization problem.
Specifically, in the case of time consuming black-box functions and when the budget of evaluations is severely limited, global optimization with random embeddings appears as a sound alternative to random search. 
Yet, in the case of box constraints on the native variables, defining suitable bounds on a low dimensional domain appears to be complex.
Indeed, a small search domain does not guarantee to find a solution even under restrictive hypotheses about the function, while a larger one may slow down convergence dramatically.
Here we tackle the issue of low-dimensional domain selection based on a detailed study of the properties of the random embedding, giving insight on the aforementioned difficulties. 
In particular, we describe a minimal low-dimensional set in correspondence with the embedded search space. 
We additionally show that an alternative equivalent embedding procedure yields simultaneously a simpler definition of the low-dimensional minimal set and better properties in practice.        
Finally, the performance and robustness gains of the proposed enhancements for Bayesian optimization are illustrated on numerical examples.

\end{abstract}


\section{Introduction}

Dealing with many variables in global optimization problems has a dramatic impact on the search of solutions, along with increased computational times, see e.g., \cite{Rios2013}.
This effect is particularly severe for methods dedicated to black-box, derivative free expensive-to-evaluate problems.
The latter are crucial in engineering, and generally in all disciplines calling for complex mathematical models whose analysis relies on intensive numerical simulations.
Among dedicated methods, those from Bayesian Optimization (BO) rely on a surrogate model to save evaluations, such as Gaussian Processes (GPs). 
They have known a fantastic development in the last two decades, both in the engineering and machine learning communities, see e.g., \cite{Shahriari2016}.
Successful extensions include dealing with stochasticity \cite{Huang2006}, variable fidelity levels \cite{Courrier2016} as well as with
constrained and multi-objective setups \cite{Feliot2015}.\\

Now, standard implementation of such algorithms are typically limited in terms of dimensionality because of the type of covariance kernels often used in practice.
The root of the difficulty with many variables is that the number of observations required to
learn a function without additional restrictive assumptions increases exponentially with the dimension,
which is known as the ``curse of dimensionality'', see e.g., \cite{Donoho2000}, \cite{Hastie2005}.
Here, we consider the optimization problem with box-constraints:
\begin{equation}
\label{eq:opt_pb}
\text{find~} \x^{**} \in \argmin \limits_{\x \in \X} f(\x)
\end{equation}
where the search domain $\X = [-1,1]^D$ is possibly of very high-dimensionality, say up to billions of variables. 
In all the rest, we suppose that $f^{**} = \min_{\x \in \X} f(\x)$ exists.\\

Recently, a groundbreaking approach was proposed in the paper \cite{Wang2013}\footnote{see \cite{Wang2016} for the extended journal version}, that
relies on random embeddings. 
Under the hypothesis that the efficient dimension of the problem is much smaller than the number of inputs,
high-dimensional optimization problems can be provably solved using global optimization algorithms in moderate dimensions,
relying on random embeddings and related results from random matrix theory. 
However implementing such algorithms in the case of bounded domains can be quite inefficient if the low-dimensional domain is not carefully chosen in accordance
with the embedding, and if space deformations caused by such embedding are not accounted for.
Here we tackle both issues, with a main focus on the choice of the low-dimensional
domain, and instantiation in the BO framework.
Before developing the main results and applications in Sections \ref{sec:two} and \ref{sec:app}, 
let us present selected state-of-the-art works, and narrow down the present aim and scope in Sections \ref{sec:oneone} and \ref{sec:onetwo} respectively.

\subsection{Related works and the random embedding approach}
\label{sec:oneone}
Many authors have focused on methods to handle high-dimensionality, in several directions. 
Selecting few variables is a rather natural idea to get back to a moderate search space, as done e.g., in \cite{Song2007}, \cite{Neal1996}, \cite{Rasmussen2006} or \cite{Chen2012}. 
Another common strategy of dimension reduction is to construct a mapping from the high-dimensional research space to a smaller one, see e.g., \cite{Viswanath2011}, \cite{Liu2014} and references therein. 
Other techniques suppose that the black-box function is only varying along a low dimensional subspace, possibly not aligned with the canonical basis,
such as in \cite{Djolonga2013} or \cite{Garnett2014}, using low rank matrix learning. 
With few unknown active variables, \cite{Carpentier2012} proposes to combine compressed sensing with linear bandits for a more parcimonious optimization.
Lastly, incorporating structural assumptions such as additivity within GP models is another angle of attack, see \cite{Durrande2010}, \cite{Durrande2011}, \cite{Duvenaud2014}, \cite{Li2016}, \cite{Wang2018} and references therein. 
They enjoy a linear learning rate with respect to the problem dimensionality, as used e.g., in \cite{Kandasamy2015} or \cite{Ivanov2014}.\\ 

In most of the above references, a significant part of the budget is dedicated to uncover the structure of the black-box, 
which may impact optimization in the case of very scarce budgets.
In \cite{Wang2013}, with the Random EMbedding Bayesian Optimization (REMBO) method, a radically different point of view was adopted, 
by simply relying on a randomly selected embedding.
Even if the main hypothesis is again that the high-dimensional function only depends on a low-dimensional subspace of dimension $d$,
the so-called \emph{low effective dimensionality} property,
no effort is dedicated to learn it.
This strong property is backed by empirical evidence in many cases, see e.g., references in \cite{Wang2013}, \cite{Constantine2014} or in \cite{Iooss2015}.\\ 

The principle is to embed a low dimensional set - usually a box - $\Y \subseteq \R^d$ to $\X$,
using a randomly drawn matrix $\A \in \R^{D \times d}$ within the mapping $\phi$: $\y \in \Y \rightarrow p_\X(\A \y) \in \X$, where $p_\X$ denotes the convex projection onto $\X$. 
By doing so, the initial search space $\X$ is reduced to a fraction (possibly the totality) of the embedded space $\Emb := \phi(\R^d)$,
which remains fixed for the rest of the process. 
The corresponding transformed optimization problem writes:
\begin{equation}
\label{eq:pb_trans}
\text{find~} \x^* \in \argmin \limits_{\x \in \Emb} f(\x)
\end{equation}
which may seem unpractical when formulated directly in terms of the $\X$ space but can be grasped more intuitively when parametrized in terms of the $\Y$ space:
$$(\pbR): \text{~find~} \y^* \in \Y \subseteq \R^d \text{~such~that~} f(\phi(\y^*)) = f(\x^*).$$
Conditions such that $f^* := f(\x^*) = f^{**}$ are addressed in \cite{Wang2013} and relaxed, e.g., in \cite{Qian2016}.
Notably, solutions coincide if the influential subspace is spanned by variables of $\X$, i.e., most variables do not have any influence.
Here we rather focus on ensuring that there exists $\y^* \in \Y$ such that $f(\phi(\y^*)) = f^*$, through the definition of $\Y$ in problem $(\pbR)$ as detailed in Section \ref{sec:onetwo}.
Remarkably, any global optimization algorithm can potentially be used for solving $(\pbR)$. In the application section we focus specifically on GP-based BO methods.  
\\

To fix ideas, Fig. \ref{fig:ex_intro} is an illustration of the random embedding principle as well as of the various sets mentioned so far. 
On this example with $D = 2$, the original function is defined on $\X = [-1,1]^2$ with a single unknown active variable ($d = 1$). 
Even if the search for the optimum is restricted to the red broken line with problem ($\ref{eq:pb_trans}$), a solution to problem $(\ref{eq:opt_pb})$ can still be found under the settings of problem $(\pbR)$.

% \begin{figure}[htpb]
% \centering
% \def\svgwidth{19.5cm}
% \import{}{ex_intro-1_mod.pdf_tex}
% \caption{Example with $d=1$ and $D = 2$. 
% The filled level lines are those of a function defined on $\X$ depending only on its second variable, whose optimal value is highlighted by the dashed line (at $x_2 \approx 0.52$).
% The image of the sets $\Y_1 = [-6, 6]$, $\Y_2 = [-1, 1]$ and $\Y^* = [-5, 5]$ by the matrix $\A = [0.5, 0.2]^\top$ are delimited on $\Ran(\A)$ with symbols.
% The extent of the orthogonal projection of $\X$ onto $\Ran(\A)$ is delimited by triangles.
% }
% \label{fig:ex_intro}
% \end{figure}

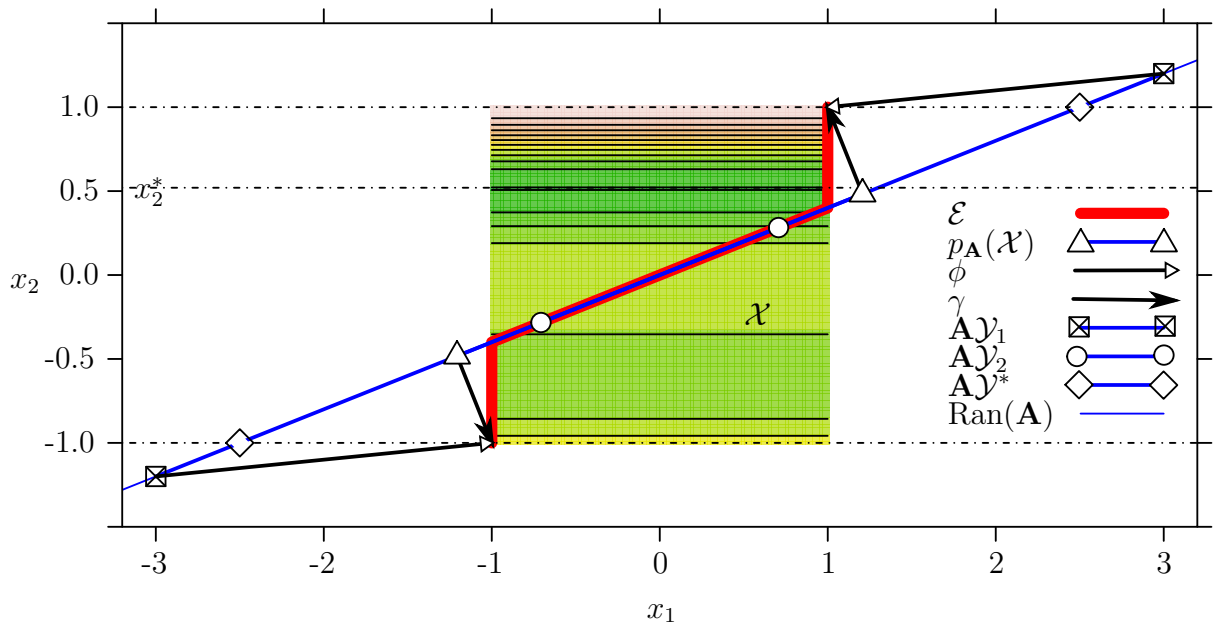
\begin{figure}[htpb]
\def\svgwidth{20cm}

\begingroup%
  \makeatletter%
  \providecommand\color[2][]{%
    \errmessage{(Inkscape) Color is used for the text in Inkscape, but the package 'color.sty' is not loaded}%
    \renewcommand\color[2][]{}%
  }%
  \providecommand\transparent[1]{%
    \errmessage{(Inkscape) Transparency is used (non-zero) for the text in Inkscape, but the package 'transparent.sty' is not loaded}%
    \renewcommand\transparent[1]{}%
  }%
  \providecommand\rotatebox[2]{#2}%
  \ifx\svgwidth\undefined%
    \setlength{\unitlength}{604bp}%
    \ifx\svgscale\undefined%
      \relax%
    \else%
      \setlength{\unitlength}{\unitlength * \real{\svgscale}}%
    \fi%
  \else%
    \setlength{\unitlength}{\svgwidth}%
  \fi%
  \global\let\svgwidth\undefined%
  \global\let\svgscale\undefined%
  \makeatother%
  \begin{picture}(1,0.41390728)%
    \put(0,0){\includegraphics[width=\unitlength]{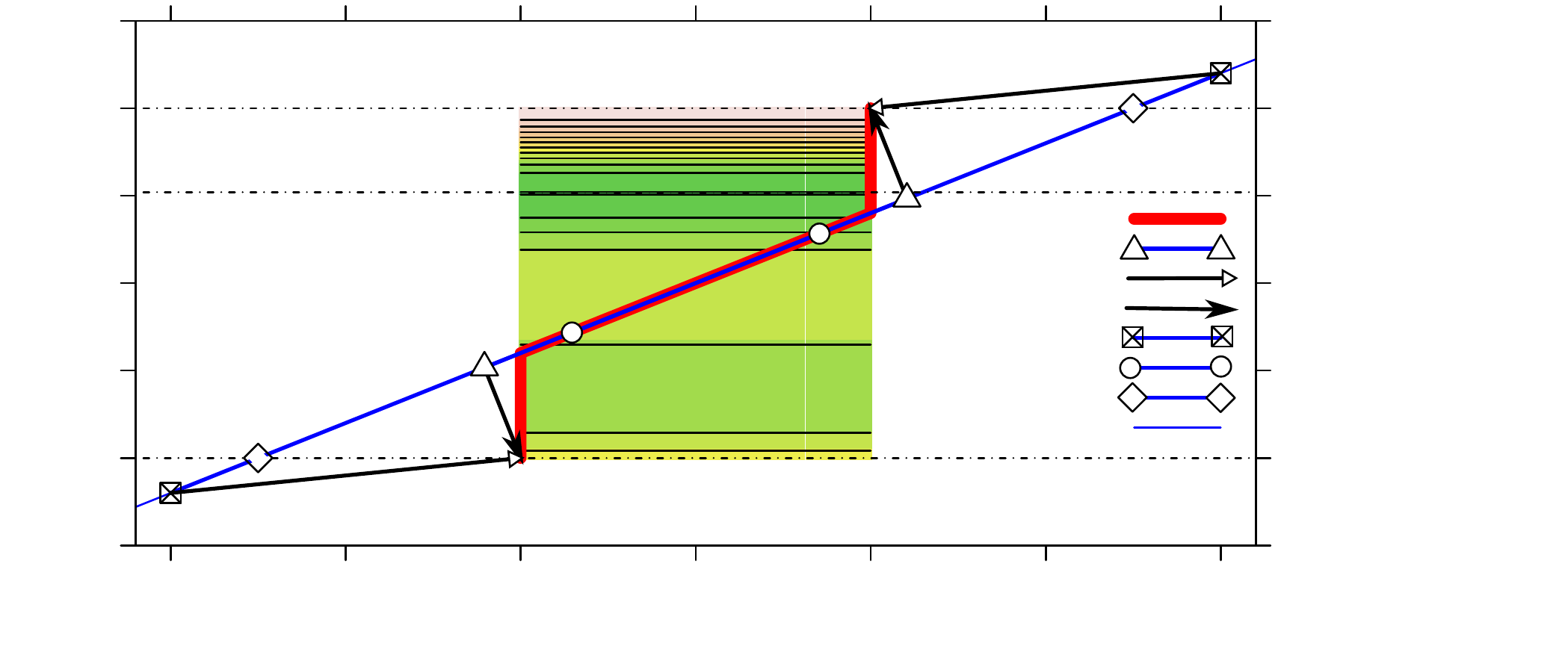}}%
    \put(0.43518212,0.00529801){\color[rgb]{0,0,0}\makebox(0,0)[lb]{\smash{$x_1$}}}%
    \put(0.01210265,0.22450331){\color[rgb]{0,0,0}{\makebox(0,0)[lb]{\smash{$x_2$}}}}%
    \put(0.5,0.2){\color[rgb]{0,0,0}{\makebox(0,0)[lb]{\smash{$\X$}}}}%
    \put(0.03506623,0.11554636){\color[rgb]{0,0,0}\makebox(0,0)[lb]{\smash{-1.0}}}%
    \put(0.03506623,0.17137417){\color[rgb]{0,0,0}\makebox(0,0)[lb]{\smash{-0.5}}}%
    \put(0.0447351,0.22721854){\color[rgb]{0,0,0}\makebox(0,0)[lb]{\smash{0.0}}}%
    \put(0.0447351,0.28304636){\color[rgb]{0,0,0}\makebox(0,0)[lb]{\smash{0.5}}}%
    \put(0.0447351,0.33887417){\color[rgb]{0,0,0}\makebox(0,0)[lb]{\smash{1.0}}}%
    \put(0.09942053,0.035){\color[rgb]{0,0,0}\makebox(0,0)[lb]{\smash{-3}}}%
    \put(0.21107616,0.035){\color[rgb]{0,0,0}\makebox(0,0)[lb]{\smash{-2}}}%
    \put(0.32273179,0.035){\color[rgb]{0,0,0}\makebox(0,0)[lb]{\smash{-1}}}%
    \put(0.43923841,0.035){\color[rgb]{0,0,0}\makebox(0,0)[lb]{\smash{0}}}%
    \put(0.55089404,0.035){\color[rgb]{0,0,0}\makebox(0,0)[lb]{\smash{1}}}%
    \put(0.66256623,0.035){\color[rgb]{0,0,0}\makebox(0,0)[lb]{\smash{2}}}%
    \put(0.77422185,0.035){\color[rgb]{0,0,0}\makebox(0,0)[lb]{\smash{3}}}%
    \put(0.09450331,0.2840894){\color[rgb]{0,0,0}\makebox(0,0)[lb]{\smash{$x_2^*$}}}%
    \put(0.63523015,0.26706036){\color[rgb]{0,0,0}\makebox(0,0)[lb]{\smash{$\Emb$}}}%
    \put(0.63523015,0.24803718){\color[rgb]{0,0,0}\makebox(0,0)[lb]{\smash{$p_\A(\X)$}}}%
    \put(0.63523015,0.22899744){\color[rgb]{0,0,0}\makebox(0,0)[lb]{\smash{$\phi$}}}%
    \put(0.63523015,0.20997426){\color[rgb]{0,0,0}\makebox(0,0)[lb]{\smash{$\gamma$}}}%
    \put(0.63523015,0.19093453){\color[rgb]{0,0,0}\makebox(0,0)[lb]{\smash{$\A\Y_1$}}}%
    \put(0.63523015,0.17189479){\color[rgb]{0,0,0}\makebox(0,0)[lb]{\smash{$\A\Y_2$}}}%
    \put(0.63523015,0.15287161){\color[rgb]{0,0,0}\makebox(0,0)[lb]{\smash{$\A\Y^*$}}}%
    \put(0.63523015,0.13383188){\color[rgb]{0,0,0}\makebox(0,0)[lb]{\smash{$\Ran(\A)$}}}%
  \end{picture}%
\endgroup%

% \import{}{ex_intro-1_mod.pdf_tex}
\caption{Example with $d=1$ and $D = 2$. 
The filled level lines are those of a function defined on $\X$ depending only on its second variable, whose optimal value is highlighted by the dashed line (at $x_2 \approx 0.52$).
The image of the sets $\Y_1 = [-6, 6]$, $\Y_2 = [-1, 1]$ and $\Y^* = [-5, 5]$ by the matrix $\A = [0.5, 0.2]^\top$ are delimited on $\Ran(\A)$ with symbols.
The extent of the orthogonal projection of $\X$ onto $\Ran(\A)$ is delimited by triangles.
}
\label{fig:ex_intro}
\end{figure}

\subsection{Motivation: Limits of random embeddings}
\label{sec:onetwo}

If the random embedding technique with problem $(\pbR)$ has demonstrated its practical efficiency in several test cases \cite{Wang2013,Binois2015a},
it still suffers from practical difficulties, related to the definition of $\Y$.
The first one, as discussed in \cite{Wang2013}, is non-injectivity due to the convex projection, i.e., distant points in $\Y$ may have the same image in $\Emb$. 
In Fig. \ref{fig:ex_intro}, this is the case with $\Y_1$, for the portions between the diamonds and crossed boxes. 
As a consequence, taking $\Y$ too large makes the search of solutions harder.
Preventing this issue is possible by several means, such as using high-dimensional distance information.
For GPs, it amounts to consider distances either on $\Emb$ within the covariance kernel as originally suggested,
or, as proposed by \cite{Binois2015a}, on the orthogonal projection $p_\A$ of $\Emb$ onto the random linear subspace spanned by $\A$, i.e., onto $\Ran(\A)$.
The advantage of the latter is to remain low-dimensional.\\

The example in Fig. \ref{fig:ex_intro} highlights another remaining difficulty, also mentioned e.g., in \cite{Kandasamy2015,Shahriari2016}:
if $\Y$ is too small (e.g., with $\Y_2$), there may be no solution to ($\pbR$).  
As of now, only empirical rules have been provided, with fixed bounds for $\Y$ (giving the too small $\Y_2$ in this example).
A possible workaround is to use several embeddings, preferably in parallel \cite{Wang2013}. 
We argue that it may also be a sound option to keep a single one, benefiting of parallelism in solving problem ($\ref{eq:pb_trans}$) instead.\\

\subsection{Contributions}

For a given random matrix $\A$, the selection of a set $\Y$ in solving problem ($\pbR$) balances between efficiency (in favor of a smaller search space), practicality (in its description) and robustness (i.e., contains a solution of ($\ref{eq:pb_trans}$)).   
Consequently, we focus on the question $(\pbQ)$ of taking $\Y$ as a set of smallest volume still ensuring that solutions of ($\pbR$) and $(\ref{eq:pb_trans})$ are equivalent,
i.e., 
$$(\pbQ): \text{find}~ \Y^* \text{ such that } \Vol(\Y^*) = \inf \limits_{\Y \subset \R^d,~\phi(\Y) = \Emb} \Vol(\Y).$$

Based on an extensive description of sets related to the mapping, we exhibit a solution of ($\pbQ$), the set $\U$,
while providing additional insight on the difficulties encountered in practice.
In Fig. \ref{fig:ex_intro}, the unique (in this $d=1$ case) optimal $\Y^* = \U = [-5,5]$ maps to the portion of $\Ran(\A)$ between the two diamonds, 
whose convex projections on $\X$ are the extremities of $\Emb$ (as is also the case for any point further away from $O$). 
Unfortunately this strategy does not generalize well in higher dimensions, where the description of $\U$ becomes cumbersome, due to the convex projection component of the mapping $\phi$.\\

A first attempt to alleviate the impact of the convex projection step proposed by \cite{Binois2015a} is to rely on an additional orthogonal projection step.
Here, we extend this approach to completely by-pass the convex projection. 
That is, instead of associating a point of $\Ran(\A)$ with its convex projection on $\X$,
we propose to associate a point of $p_\A(\Emb) \subset \Ran(\A)$ with its pre-image in $\Emb$ by $p_\A$, informally speaking \emph{inverting} the orthogonal projection -- \emph{back-projecting} in short. 
Expressed in a basis of $\Ran(\A)$, denoted by the $d \times D$ matrix $\B$, coordinates are $d$-dimensional.
Then the embedding procedure is still a mapping between $\R^d$ and $\Emb$, denoted by $\gamma: \B p_\A(\Emb) \subset \R^d \to \Emb$. 
Most of the present work is dedicated to study the validity and applicability of this back-projection.
The corresponding alternative formulations of $(\pbR)$ and $(\pbQ)$ write:
$$(\pbR'): \text{~find~} \y^* \in \Y \subseteq \R^d \text{~such~that~} f(\gamma(\y^*)) = f^*$$
and 
$$(\pbQ'): \text{find}~ \Y^* \text{ such that } \Vol(\Y^*) = \inf \limits_{\Y \subset \R^d,~\gamma(\Y) = \Emb} \Vol(\Y).$$
The benefits of these formulations are, first, that a solution of $(\pbQ')$ is by construction $\Y^* = \Z := \B p_\A(\X)$
and, second, that they enjoy better properties from an optimization perspective.
Back to Fig. \ref{fig:ex_intro}, $p_\A(\X)$ is delimited by the two triangles.\\

Finally, these enhancements are adapted for Bayesian optimization, via the definition of appropriate covariance kernels. 
They achieve significant improvements, with comparable or better performance than the initial version of REMBO on a set of examples, while discarding the risk of missing the optimum in $\Emb$.\\

The remainder of the article is as follows.
In the subsequent Section \ref{sec:two} are presented our main results towards question ($\pbQ$), both in its original setup and in the alternative one ($\pbQ'$).
A main contribution of this work is to explicitly write the sets $\U$, $\Z$, and the alternative mapping $\gamma$, depending on the matrix $\A$.
While these results are of general interest for global optimization with random embedding regardless of the base optimization algorithm,
Section \ref{sec:app} is dedicated to the particular case of Bayesian optimization with random embeddings on various experiments. Section \ref{sec:ccl} concludes the article.

%%%%%%%%%%%%%%%%%%%%%%%%%%%%%%%%%%%%%%%%%%%%%%%%%%%%%%%%%%%%%%%%%%%%%%%%%%%%%%%

%%%%%%%%%%%%%%%%%%%%%%%%%%%%%%%%%%%%%%%%%%%%%%%%%%%%%%%%%%%%%%%%%%%%%%%%%%%%%%%%
\section{Minimal sets solutions to questions \texorpdfstring{($\pbQ$)}{(Q)} and \texorpdfstring{($\pbQ'$)}{(Q')}}
\label{sec:two}

Throughout this section, we consider that the matrix $\A$ is given, and for simplicity that it belongs to the following class of matrices, denoted $\classA$:
$$\classA = \set{\A \in \R^{D \times d} \text{~such~that~any~} d \times d \text{~extracted~submatrix~is~invertible~(i.e.,~of~rank~$d$)}}.$$
This mild condition is ensured with probability one for random matrices with standard Gaussian i.i.d.\ entries, as used in \cite{Wang2013}.
Before discussing the relative merits of problems ($\pbR$) and ($\pbR'$), we start by exhibiting sets of interest in $\Y$, $\X$ and $\Ran(\A)$ in both cases.

\subsection{A minimal set in \texorpdfstring{$\R^d$}{Rd} mapping to \texorpdfstring{$\Emb$}{E}}
\label{ssec:U}

Until now, the description of the set $\Y$ has been relatively vague -- \cite{Wang2016} states that there is room for improvement.
This question is settled in \cite{Wang2013} by setting $\Y = [-\sqrt{d}, \sqrt{d}]^d$, while \cite[Theorem 3]{Wang2013} only ensures to find a solution to problem $(\pbR)$ with probability 1 in the particular case where $\Y = \R^d$. 
On the other hand, the sets $\Emb$ and $p_\A(\X)$ are fixed and well defined given $\A$. 
This motivates us to describe a new set $\U \subset \R^d$, containing a solution of $(\pbR)$, of minimal volume, and that can also be described from $\A$.\\

To this end, consider the low-dimensional space $\R^d$. Denote by $H_{\textbf{a}, \delta}$ the hyperplane in $\R^d$ with normal vector $\textbf{a} \in \R^d$ and offset $\delta \in \R$: $H_{\textbf{a}, \delta} = \set{\y \in \R^d \telque \langle \textbf{a} , \y \rangle = \delta}$.  
Our analysis in the low dimensional space begins by a general definition of %what we call
\emph{strips}.

\begin{Definition}
We call \emph{strip} with parameters $a \in \R^d$ and $\delta \in \R$, denoted by $\Band_{\textbf{a}, \delta}$ the
 set of points between the parallel hyperplanes $H_{\textbf{a},-\delta}$ and $H_{\textbf{a},\delta}$: $\Band_{\textbf{a}, \delta} = \set{\y \in \R^d \telque | \langle \textbf{a}, \y \rangle | \leq |\delta|}$.
\end{Definition}

Let us now consider hyperplanes with normal vectors given by the rows of a matrix $\A \in \R^{D\times d}$ and with fixed $\delta = 1$. The $D$ corresponding strips, now simply denoted $\Band_i$, are given by:
$$\Band_i = \left\{\y \in \R^d, -1 \leq \A_i \y \leq 1 \right\}.$$
The intersection of all strips $\Band_i$ is denoted $\I$. It corresponds to the pre-image of $\X \cap \Ran(\A)$ by $\A$:
 $$\I = \bigcap \limits_{i=1}^D \Band_i  = \set{\y \in \R^d, \forall i = 1, \dots, D: -1 \leq \A_i \y \leq 1} = \set{\y \in \R^d \telque p_\X(\A \y) = \A \y}.$$
 
 Of interest will also be intersections of $d$ strips, corresponding to \emph{parallelotopes}, i.e., linear transformations of a $d$-cube in a $d$-dimensional subspace, see e.g., \cite{Le2013}. In particular, using set notations $I = \set{i_1, \dots, i_d} \subseteq \set{1, \dots, D}$, denote the parallelotope with strips $I$
 $$\Paral_I = \set{\y \in \R^d, \forall i \in I: -1 \leq \A_i \y \leq 1} = \bigcap \limits_{i \in I} \Band_i .$$
 
There are $\binom{D}{d}$ different parallelotopes $\Paral_I$. We thus consider their union, which is referred to as $\U$:
$$\U = \bigcup \limits_{I \subseteq \{1, \dots, D\}, |I| = d} \Paral_I$$
where $|I|$ is the size of $I$. In fact we show in the following Theorem \ref{prop:ZH} that $\U$ is the smallest closed set such that the map $\restriction{\phi}{\U}: \U \to \Emb$, $\y \mapsto p_\X(\A \y)$ is surjective.

\begin{theorem}
\label{prop:ZH}
If $\A \in \classA$, then $\U$ is the smallest closed set $\Y \subseteq \R^d$ such that $p_\X(\A \Y) = \Emb$. Furthermore, $\U$ is a compact and star-shaped set with respect to every point in $\I$.
\end{theorem}

The detailed proof is can be found in Appendix \ref{ap:proofth1}.\\

In other words, if we choose $\Y \supseteq \U$, then a solution to problem $(\pbR)$ and $(\ref{eq:pb_trans})$ can be found.
Yet, from its definition as a union of intersections, the set $\U$ is unpractical to directly work with. 
Indeed, it is more common practice to work on simpler sets such as boxes instead of star-shaped sets.  
Based upon \cite[Theorem 3]{Wang2013}, their choice of $\Y = [-\sqrt{d}, \sqrt{d}]^d$ originates from results on the radius of a parallelotope $\Paral_I$,
corresponding to a probability greater than $1 - \varepsilon$ of containing a solution, with $\varepsilon = \log(d)/\sqrt{d}$ (see the discussion in \cite{Wang2016}).
Unfortunately, there is no such result for any matrix in general nor for the maximum radius of a parallelotope in which one could wish to enclose $\U$. 
Yet, it is still easy to detect whether a given point is in $\U$ or not: if no more than $d$ components of $\A \y$ are superior to one in absolute value. 
Hence selecting a large $\Y$ such that $\U \subseteq \Y$ is always possible,
even though it may prove to be extremely large and counterproductive.\\ 

As a by-product, the proof of Theorem \ref{prop:ZH} gives a possibility to find pre-images in $\Y$ for elements of $\Emb$:
letting $\x \in \Emb$, pre-image(s) in $\U$ are solutions of the following system of linear (in)equations: 
$$\text{find}~\y \in \U \text{ s.t. }\left\{ 
\begin{array}{ll}
\A_J \y&=  \x_J \\
\A_K \y &\geq \mathbf{1}_{|K|}\\
\A_L \y &\leq  -\mathbf{1}_{|L|} 
\end{array}
\right.
$$
where $J$, $K$ and $L$ are the sets of components of $\x$ such that $|x_i| \leq 1$, $x_i > 1$ and $x_i < -1$ respectively, with $\mathbf{1}_{|K|} = (1, \dots, 1)^\top$ of length $|K|$.  
A solution to this problem exists by Theorem \ref{prop:ZH}, several may even exist if $|J| < d$.\\

To sum up, we have highlighted three different sets of interest: parallelotopes $\Paral_I$, the intersection of all of them $\I$, and their union $\U$. The sets $\U$, $\Paral_I$ and $\I$ are illustrated with $d = 2$ in Fig. \ref{fig:ZH}. On the top figures, (a), strips are marked by lines. 
Next, we conduct a similar analysis of with the mapping $\gamma$.

%%%%%%%%%%%%%%%%%%%%%%%%%%%%%%%%%%%%%%%%%%%%%%%%%%%%%%%%%%%%%%%%%%%%%%%%%%%%%%%%%%%%%%%%%%
\subsection{Bijection between \texorpdfstring{$\Emb$}{E} and \texorpdfstring{$\Z$}{Z}}
\label{ssec:Z}

The core idea behind formulations $(\pbR')$ and $(\pbQ')$ is, through using $\Ran(\A)$ as low-dimensional domain, 
to replace the convex projection by an inversion of the orthogonal projection, 
\emph{in fine} replacing the mapping $\phi$ by another, $\gamma$, with better properties.
It is worth insisting that the search for a minimum occurs on the same set $\Emb$ in the original high-dimensional domain $\X$.\\

A core set here is the one obtained by projection of $\X$ onto $\Ran(\A)$, which is known to be a zonotope, a special class of convex centrally symmetric polytopes, see Definition \ref{def:zon} and e.g., in \cite{McMullen1971}, \cite{Ziegler1995} or \cite{Le2013}.

\begin{Definition}[Zonotope as hypercube projection, adapted from \cite{Le2013}]
\label{def:zon}
A (centered) $D$-zonotope in $\R^d$ is the image of the $[-1,1]^D$ hypercube $\X$ by a linear mapping. Given a matrix $\B \in \R^{d \times D}$ representing the linear mapping, the zonotope $\mathcal{Z}$ is defined by $\mathcal{Z} = \B \X$.
\end{Definition}

This representation of a zonotope is known as its \emph{generator} representation, 
while it can also be described by vertices enumeration or hyperplane intersections like any other convex set, see e.g., \cite{Krein1940}. 
They provide a very compact representation of sets, which is useful for instance in set estimation, see e.g., \cite{Le2013}.\\ 

In the following, we assume that rows of $\B$ form an orthonormal basis of $\Ran(\A)$ in $\R^D$, i.e., $\B \B^\top = \Id_d$, with $\Id_d$ the identity matrix of size $d\times d$. 
The orthogonal projection onto $\Ran(\A)$ then simply writes: $p_\A = \B^\top \B$ \cite{Meyer2000}.
Let us point out that without the orthonormality condition, expressions involve pseudo-inverses.
Now, as we aim to define a mapping between $\Z$ and $\Emb$, a key element given by Proposition \ref{prop:eq_zon} is that the orthogonal projection of the set $\Emb$ onto $\Ran(\A)$ 
actually coincides with the one of $\X$ onto $\Ran(\A)$, i.e., $\B^\top \Z$. 
It thus inherits the properties of a zonotope.

\begin{Proposition}

$p_\A(\X) = p_\A(\Emb)$, or equivalently, $\B \Emb = \B\X = \Z$.
\label{prop:eq_zon}
\end{Proposition}
\begin{proof}
Please refer to Appendix \ref{ap:proof2}.
\end{proof}

To provide some intuition about the ideas and sets involved in this Section \ref{ssec:Z}, compared to those of Section \ref{ssec:U}, they are illustrated with an example in Fig. \ref{fig:ZH}, panel (d).\\

\begin{figure}%
\centering%
\begin{subfigure}[t]{\linewidth}%
\centering%
\includegraphics[width=0.5\textwidth]{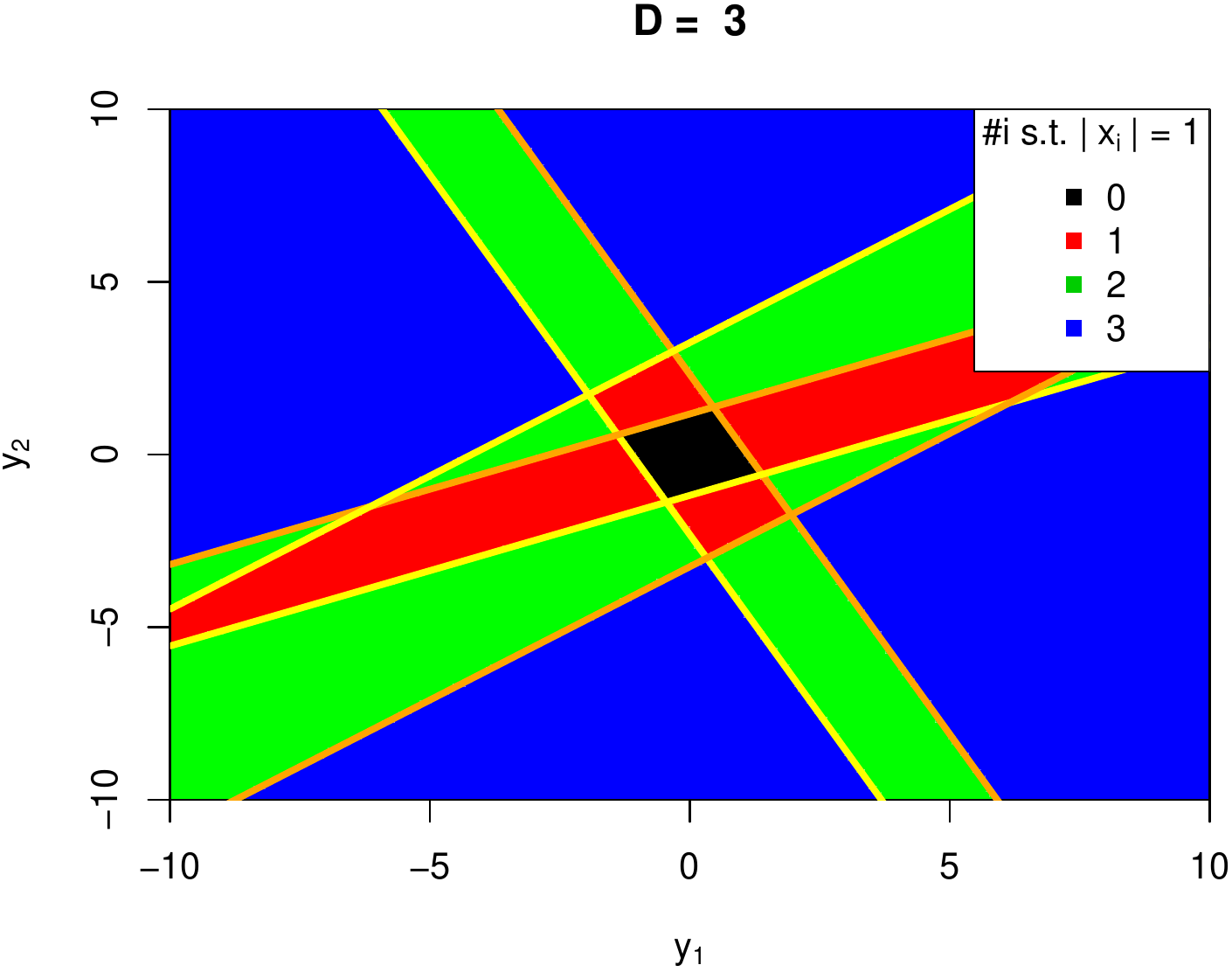}%
\includegraphics[width=0.5\textwidth]{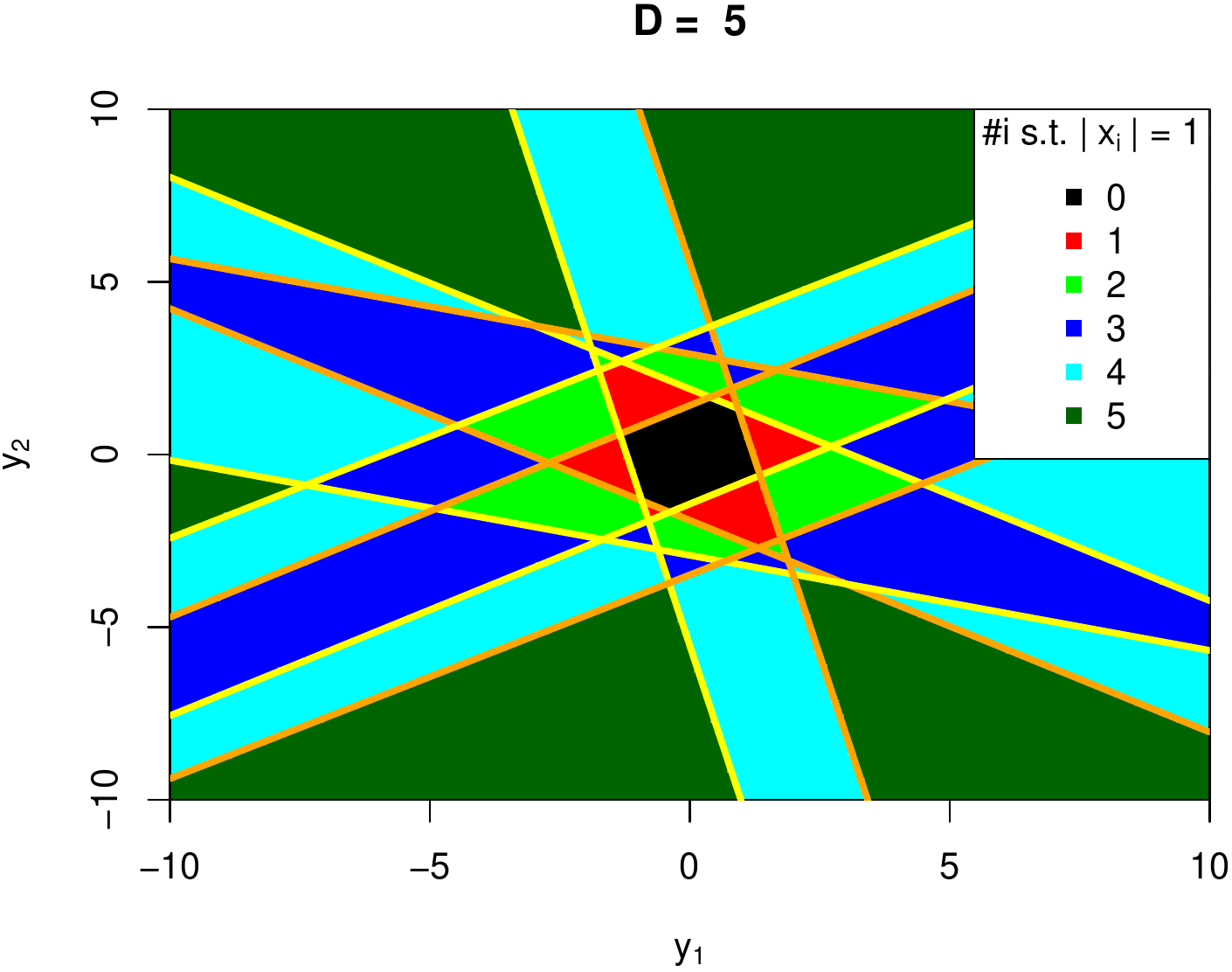}%
\caption{Strips $\Band_i$ in $\R^2$, formed by the orange and yellow lines standing for the parallel hyperplanes.}
\end{subfigure}\\
\begin{subfigure}[c]{0.48\linewidth}%
\centering%
\includegraphics[width=0.8\columnwidth, trim = {0.8cm, 1.1cm, 0.8cm, 1.2cm}, clip]{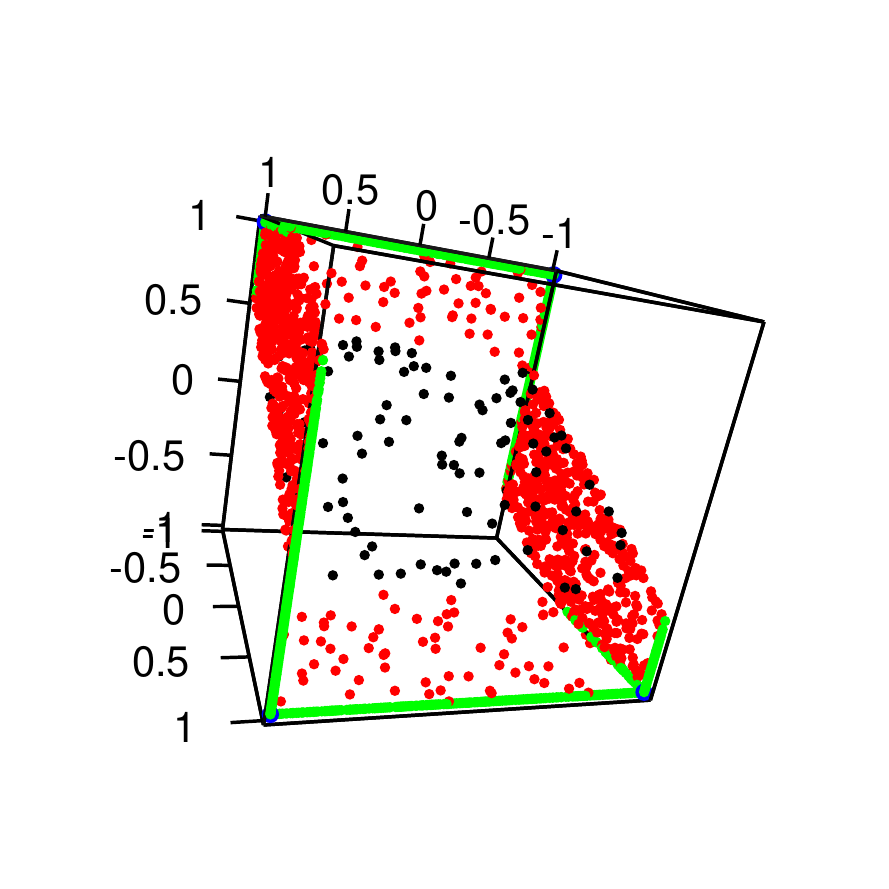}%
\caption{$\Emb$ in $\R^3$.}
\end{subfigure}%
\begin{subfigure}[c]{0.48\linewidth}%
\centering%
\includegraphics[width=\textwidth]{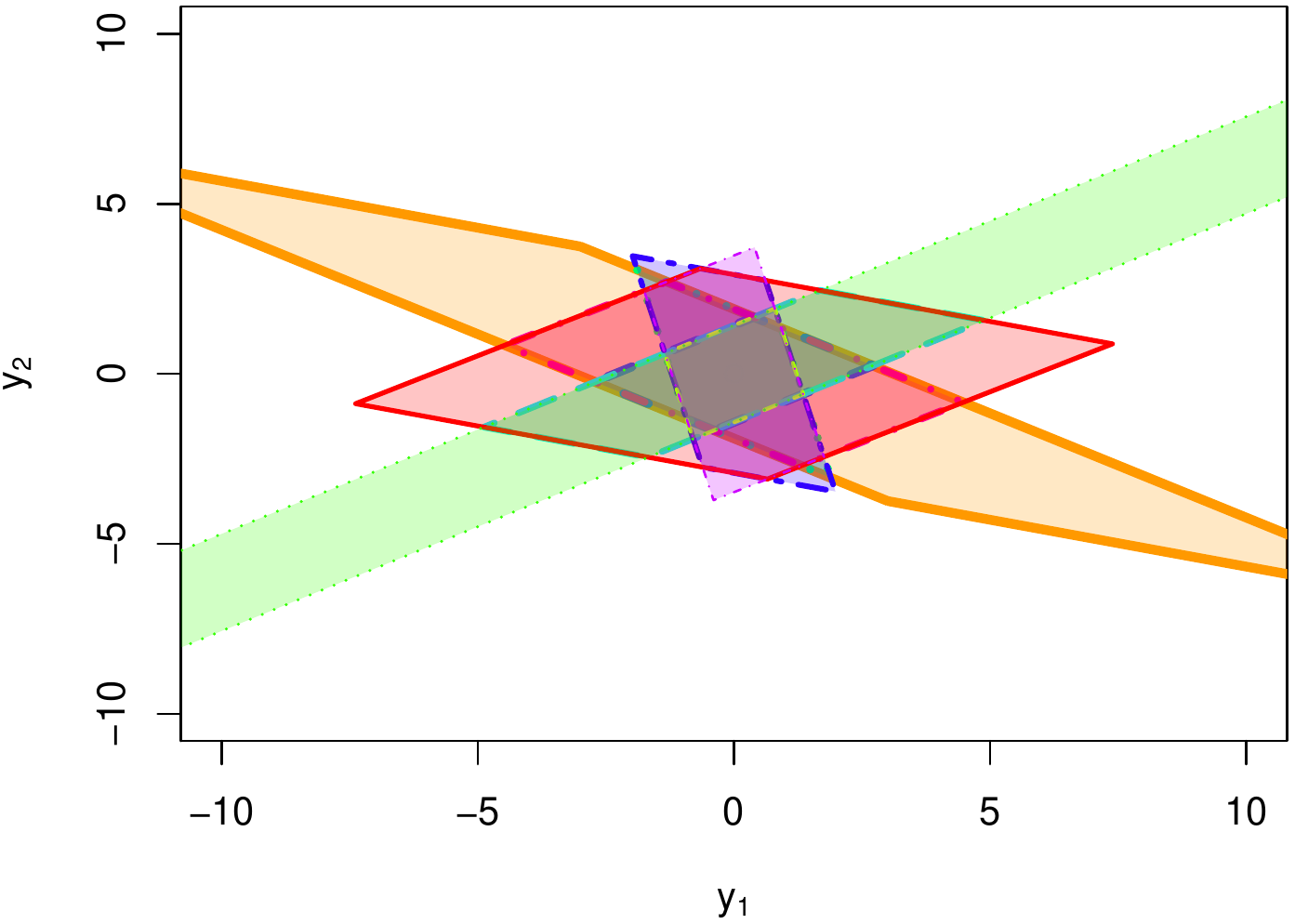}
\caption{Parallelotopes in $\R^2$ forming $\U$.} 
\end{subfigure}\\
\begin{subfigure}[b]{\linewidth}%
\centering%
\includegraphics[width=0.5\textwidth]{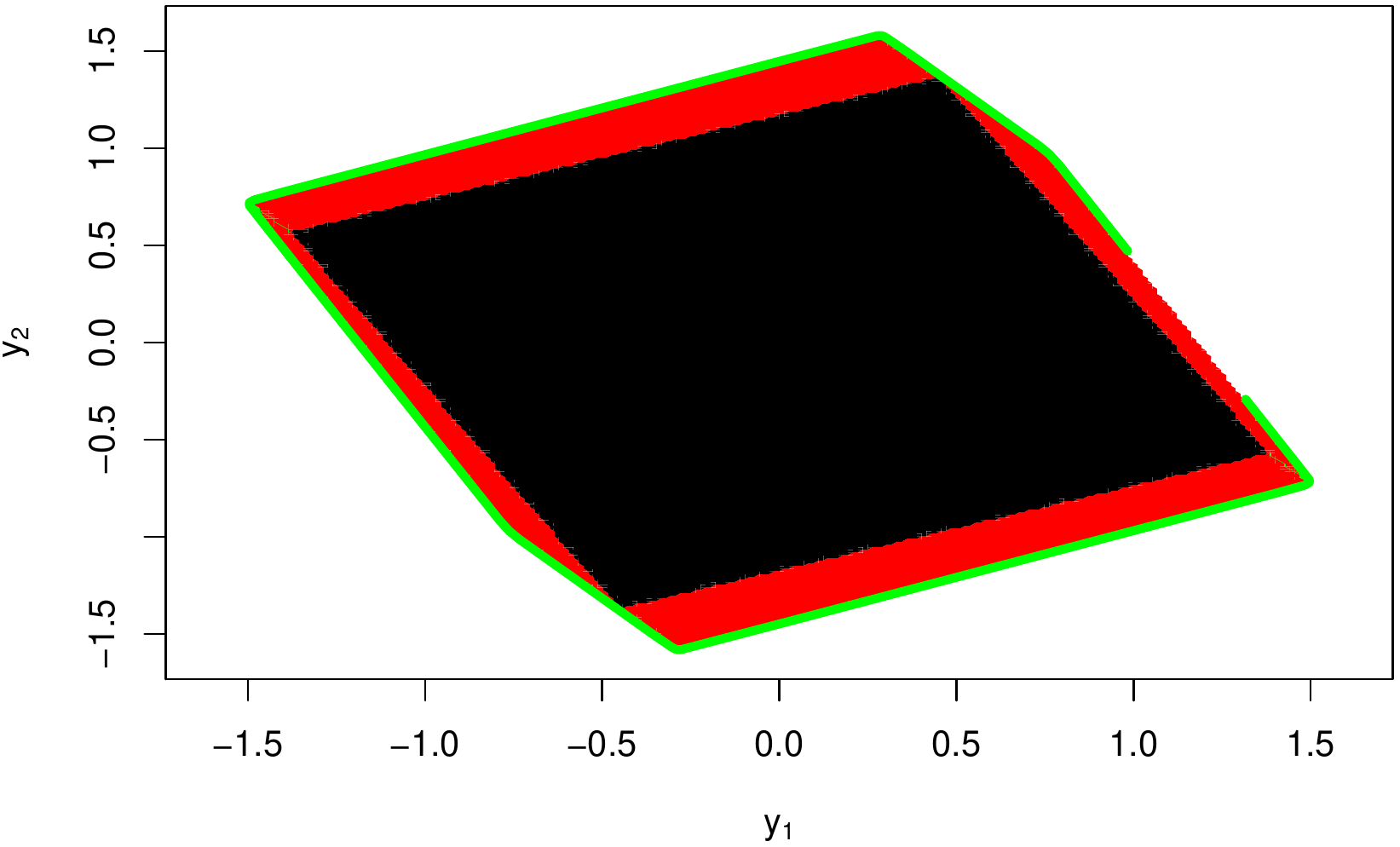}%
\includegraphics[width=0.5\textwidth]{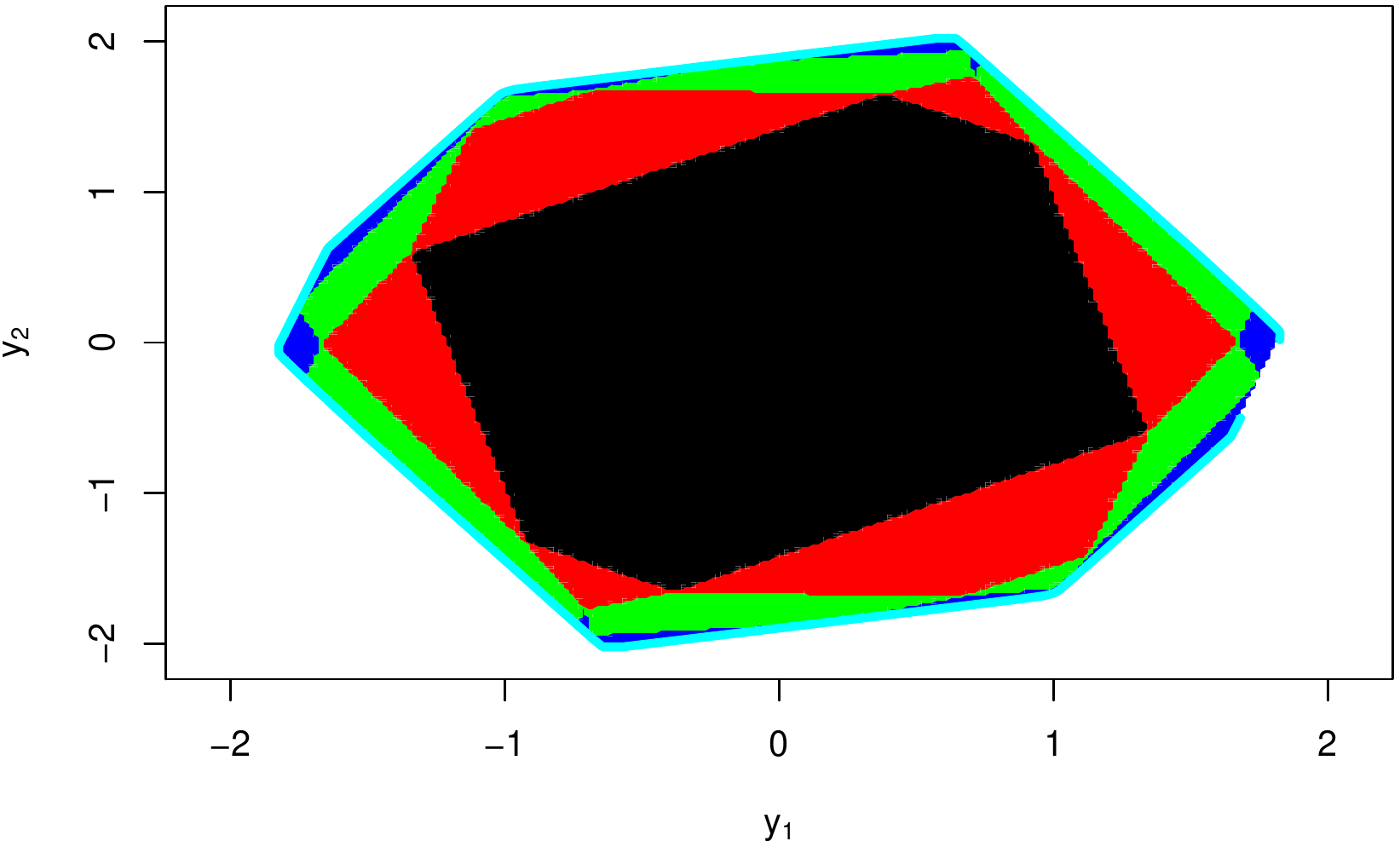} 
\caption{Zonotopes $\Z$ in $\R^2$.} 
\end{subfigure}%
\caption{Representation of the sets of interest introduced in Section \ref{sec:two} with $d =2$ and $D = 3$ (left) or $D = 5$ (right) for fixed $\A$ matrices. Except for panel (c), colors represents how many variables in $\X$ are not in $]-1, 1[$. The intersection of all parallelograms $\I$ is in black.
Panel (a) highlights strips in $\Y$ when using the mapping $\phi$. Panel (b) is the image by $\phi$ of the top left panel, in $\X$.  
Panel (c): each of the 10 parallelotopes (here parallelograms) $\Paral_I$ is depicted with a different color and their union is the minimal set $\U$. For illustration purpose, $\U$ is truncated: the cut green parallelogram is approximately enclosed in $[-210, 210] \times [130,130]$. Panel (d): zonotopes used as pre-images with mapping $\gamma$.
}%
\label{fig:ZH}%
\end{figure}%

Now, the difficulty for the associated mapping is to \emph{invert} the orthogonal projection of $\Emb$ onto $\Ran(\A)$,
more precisely onto an orthonormal basis $\B$ of the latter.
One way to perform this task is to define $\gamma(\y): \Z \to \R^D$ as the map that, 
first linearly embeds $\y \in \Z$ to $\R^D$ with the matrix $\B^\top$ and then maps this $\B^\top \y \in \R^D$ to the closest point $\x \in \X$ whose orthogonal projections onto $\Ran(\A)$, i.e., $\B\x$, coincide with $\B^\top\y$. 
This is represented in Fig. \ref{fig:ex_intro} with arrows, in the illustrative case $d=1$, $D = 2$.\\

Now let us show that the map $\gamma$ is well defined.
Let $p_\A^{-1}(\veca) = \left\{\x \in \R^D \text{ s.t. } p_\A(\x) = \veca \right\}$ the set of pre-images of $\veca \in \R^D$ for the orthogonal projection onto $\Ran(\A)$.
Then, $$\gamma(\y) = p_{\X \cap p_\A^{-1}(\B^\top \y)} (\B^\top \y)$$ 
is the convex projection on the convex set $\X \cap p_\A^{-1}(\B^\top \y)$. Since $\y \in \Z$, $\X \cap p_\A^{-1}(\B^\top \y) \neq \emptyset$ and $\gamma$ is defined. 
This leads to the counterpart of Theorem \ref{prop:ZH} with the mapping $\gamma$.

\begin{theorem}
\label{th:2}
$\Z$ is the smallest closed set $\Y \subseteq \R^d$ such that $\gamma(\Y) = \Emb$. Furthermore, $\Z$ is a compact, convex and centrally symmetric set. 
\end{theorem}
\begin{proof}
Please refer to Appendix \ref{ap:proof3}.
\end{proof}

In practice, $\gamma$ can be written as the solution of the following quadratic programming problem:
\begin{align*}
\gamma(\y) = \arg \min \limits_{\x \in \X} \|\x - \B^\top \y\|^2 \\
s.t.~ \B\x = \y.
\end{align*}

Next we discuss the relative merits of both solutions sets $\U$ and $\Z$, with mappings $\phi$ and $\gamma$ respectively. 

\subsection{Discussion}

Recall that there is a compromise between practical implementation, size of the domain (related to convergence speed) and risk of missing a solution.
The original REMBO, with $\phi$ and $\Y = [-\sqrt{d}, \sqrt{d}]^d$ is computationally efficient with its mapping procedure and a fixed definition of the search space.
The balance between the two other points depends on the sampled matrix $\A$, and is expected to favor small domains \cite{Wang2013}.\\

If, as with questions ($\pbQ$) and ($\pbQ'$), emphasis is on ensuring that there is a solution in the low dimensional search space,
and even if both sets $\U$ and $\Z$ are compacts, $\Z$ has several advantages.
First, it is a convex set instead of a star-shaped one, with a generator description instead of a combinatorially demanding union description. 
Enclosing $\U$ in a box or a sphere requires finding the radius of the largest parallelope that enclose it, which is combinatorially difficult.
Of interest here, the smallest box enclosing $\Z$ has a simple expression: the extreme value in the $i^\mathrm{th}$ direction is $\sum\limits_{j = 1}^D |B_{i,j}|$, $1 \leq i \leq d$, see e.g., in \cite{Le2013}.
Hence it is possible to work in $$\boxx = \left[-\sum\limits_{j = 1}^D |B_{1,j}|, \sum\limits_{j = 1}^D |B_{1,j}|\right] \times \dots \times \left[-\sum\limits_{j = 1}^D |B_{d,j}|, \sum\limits_{j = 1}^D |B_{d,j}| \right].$$
Testing whether or not a given point $\y \in \R^d$ is in $\Z$ amounts to verify whether or not the linear system $\B \x = \y$ has a solution in $\X$; 
more conditions such as to identify the boundary of $\Z$ are given e.g., in \cite{Cerny2012}. 
As for $\U$, it amounts to verify that at least $d$ variables of $\A \y$ are in $[-1,1]$.\\

Even if $\gamma$ requires solving a quadratic programming problem, the additional cost usually fades in the case of expensive black-box simulators, with limited budget of evaluations. 
Finally, additional advantages of $\Z$ over $\U$ in practice are provided in Proposition \ref{prop:vols}.

\begin{Proposition}
\label{prop:vols}
Denote $\Vol_d$ the $d$-volume in $\R^D$. Let $\A \in \R^{D \times d}$ with orthonormal columns, i.e., $\A^\top \A = \Id_d$.  
Then $\Vol_d(\A \U) \geq \Vol_d(\Emb) \geq \Vol_d(\A \Z)$.
If, in addition, rows of $\A$ have equal norm, then $\Vol(\Z)/\Vol(\I) \leq d^{d/2}$, which does not depend on $D$.  
\end{Proposition}
\begin{proof}
Please refer to Appendix \ref{ap:proof4}.
\end{proof}

This result provides us with hints on the deformation of the search space occurring with both mappings, sharing the same invariant space $\I$.
With $\phi$, the volume of $\A\U \setminus \A \I$ is relatively bigger than the one of $\Emb \setminus \A \U$. 
The effect is reversed for $\gamma$. 
Since $\I$ is the set where all variables, especially the relevant ones, are not fixed\footnote{Outside of the set corresponding to $\Paral_I$ in $\R^d$, these influential variables $I$ would be fixed to $\pm 1$.},
focusing more on these is, arguably, beneficial.
In fact, the second part of the result indicates that the relative volume of the undeformed part within $\Z$ does not depend on $D$.
In preliminary tests, we tried to extend the domain to $\U$ with the mapping $\phi$, and the performance degraded considerably.
In particular, the volume of $\I$ became quickly negligible compared to the one of $\U$, when increasing $d$ or $D$.
We next discuss these points in more details for GP-based BO, illustrating these differences empirically.

\begin{Remark}
In \cite{Wang2013}, $\A$ is a random matrix with i.i.d.\ standard Gaussian entries. In this case, many results about determinants, eigenvalues and limiting distributions are known, see e.g., \cite{Vershynin2010} or \cite{Nguyen2014} as starting points into this rich literature. 
One result related to Proposition \ref{prop:vols} is that $\| \frac{1}{D} \A^\top \A - \Id_d \| \to 0$ almost surely as $d/D$ goes to 0, see e.g., \cite{Vershynin2010}. 
There are several alternatives to Gaussian random matrices, such as random matrices whose rows are randomly selected on the sphere -- corresponding to random matrices with independent rows of equal norm -- that have the same asymptotic properties \cite{Vershynin2010}.
Their use has been studied in \cite{Binois2015b}, showing benefits mostly for small $d$.
\label{rem:choice}
\end{Remark}

%%%%%%%%%%%%%%%%%%%%%%%%%%%%%%%%%%%%%%%%%%%%%%%%%%%%%%%%%%%%%%%%%%%%%%%%%%%%%%%%%%%%%%%%%%%%%%%%%%%%%
\section{Application to Bayesian optimization}
\label{sec:app}

The random embedding paradigm incorporates seamlessly within GP-based BO methods through the covariance kernel.
After a brief description of Bayesian optimization using Gaussian processes and the specific choices of covariance kernels, 
we present results on a set of test cases.

\subsection{Modified REMBO procedure}

Bayesian optimization, and especially seminal works on the expected improvement initiated in \cite{Mockus1978}, is built on two key concepts: the first one is to consider the underlying black-box function as random and to put a prior distribution that reflects beliefs about it. 
New observations are used to update the prior, resulting in a posterior distribution. 
The second pillar is an acquisition function that selects new locations using the posterior distribution to balance exploitation of promising areas and exploration of less known regions.\\

One such example is the widely used EGO algorithm \cite{Jones1998}. 
Its prior distribution is a Gaussian process, and its acquisition function the Expected Improvement (EI) criterion.
Other popular surrogate models include radial basis functions, random forests and neural networks, see e.g., \cite{Gutmann2001,Hutter2011,Krityakierne2016,Chen2016} and references therein.
As for alternative acquisition functions, we also mention those relying on an information gain as in \cite{Villemonteix2009,Hennig2012}.
The reader interested in these variations on BO may refer to \cite{Shahriari2016} for a recent review.\\ 

GP priors are attractive for their tractability since they depend only on a mean $\mu(\cdot)$ and covariance function $k(\cdot, \cdot)$. 
Assuming that $\mu$ and $k$ are given, conditionally on $n$ observations of $f$, $\f_{1:n} := (f(\x_1), \dots, f(\x_n))$, the predictive distribution is another GP, with mean and covariance functions given by:
\begin{align}
m_n(\x) & = \mu(\x) + \textbf{k}(\x)^\top \textbf{K}^{-1} (\f_{1:n} - \mu(\x_{1:n}))\\
k_n(\x, \x') & = k(\x, \x') - \textbf{k}(\x)^\top \textbf{K}^{-1} \textbf{k}(\x')
\end{align}
where $\x_{1:n} := (\x_1, \dots, \x_n)^\top$, $\textbf{k}(\x) := (k(\x, \x_1), \dots, k(\x, \x_n))^\top$ and $\textbf{K} := (k(\x_i, \x_j))_{1 \leq i,j \leq n}$ are the vector of covariances of $Y(\x)$ with the $Y(\x_i)$'s and the covariance matrix of $Y(\x_{1:n})$, respectively.
The choice of the mean and covariance function dictates the expected behavior of $f$. 
Commonly, $\mu$ is supposed to be constant while $k$ belongs to a parametric family of covariance functions such as the Gaussian and Mat\'ern kernels, corresponding to different hypothesis about the derivability of $f$. Associated hyperparameters are frequently inferred based on maximum likelihood estimates, see e.g., \cite{Rasmussen2006} or \cite{Roustant2012} for specific details.\\

In the case of EI, the improvement is defined as the difference between the current minimum of the observations and the new function value (thresholded to zero). The latter being given by the GP model, EI is the conditional expectation of the improvement at a new observation, which has a closed form expression, see e.g., \cite{Jones1998}. 
Notice that optimizing EI may be a complicated task in itself, see e.g., \cite{Franey2011}, due to multi-modality and plateaus.
Yet evaluating EI is inexpensive and off-the-shelf optimization algorithms can be used (possibly relying on derivatives).\\     

Adapting the framework of Bayesian optimization to incorporate a random embedding amounts to optimize the acquisition function on $\Y$, while evaluations are performed on $\Emb$. 
In terms of GP modeling, when using stationary covariance kernels, what matters is the distance between points. 
Several options are possible to account for high-dimensional distances through compositions of kernels with functions, also known as warpings. 
Existing warpings for $k$ defined on $\Y$ include:
\begin{itemize}
\item identity warping: distances are distances in $\Y$, the corresponding kernel is denoted $k_\Y$ in \cite{Wang2013}; 
\item random embedding and convex projection warping, i.e., using $\phi$, denoted $k_\X$ in \cite{Wang2013}; 
\item an additional composition is proposed by \cite{Binois2015a}, with orthogonal projection onto $\Ran(\A)$ and a distortion. The distortion is used to counteract the effect of the orthogonal projection on high dimensional distances: the further away from $\Ran(\A)$, the closer to the center the projection is. The warping $\Psi$ writes $\Psi(\y) = \left(1 + \frac{\| \phi(\y) - \z' \|}{\| \z' \|} \right) \z'$ with $\z' = \z / \max(1, \max\limits_{1 \leq i \leq D} | z_i |)$, $\z = p_\A(\phi(\y))$.    
\end{itemize}

With the alternative mapping $\gamma$, $k_\Y$ is defined based on distances in $\Z$, while $k_\X$ makes use of $\gamma$ instead of $\phi$. 
As for $k_\Psi$, the orthogonal projection is already performed and it only amounts to applying the correction: 
$\Psi'(\y) = \left(1 + \frac{\| \gamma(\y) - \z' \|}{\| \z' \|} \right) \z'$ with $\z' = \z / \max(1, \max\limits_{1 \leq i \leq D} | z_i |)$, $\z = \B^\top \y)$.
For the sake of readability, consider an isotropic \emph{squared-exponential} (or \emph{Gaussian}) kernel $k(\x, \x') = \exp(-\| \x -\x' \|^2/\theta^2)$ where $\theta$ is the scale or "lengthscale" hyperparameter,
Table \ref{tab:kernels} summarizes the expressions for all combinations of warping and mappings.

\begin{table}[htpb]
\centering
\caption{Gaussian kernel expressions depending on the embedding and warping. 
The first column summarizes existing kernels in the literature \cite{Wang2013,Binois2015a} relying on $\phi$, the second their transposition when using $\gamma$.}
\label{tab:kernels}
\begin{tabular}{c|c|c}
& mapping $\phi$  ($\y, \y' \in \Y$) & mapping $\gamma$ ($\y, \y' \in \Z$) \\ \hline
$\R^d$     &   $k_\Y(\y, \y') = \exp(-\| \y -\y' \|^2/\theta^2)$                  &    $k_\Y(\y, \y') = \exp(-\| \y -\y' \|^2/\theta^2)$              \\ \hline
$\Emb$     &   $k_\X(\y, \y') = \exp(-\| \phi(\y) - \phi(\y') \|^2/\theta^2)$     &    $k_\X(\y, \y') = \exp(-\| \gamma(\y) - \gamma(\y') \|^2/\theta^2)$             \\ \hline
$\Ran(\A)$ &   $k_\Psi(\y, \y') = \exp(-\| \Psi(\y) - \Psi(\y') \|^2/\theta^2)$   &    $k_\Psi'(\y, \y') = \exp(-\| \Psi'(\y) - \Psi'(\y') \|^2/\theta^2)$           
\end{tabular}
\end{table}

To take into account that $\gamma$ is not defined outside of $\Z$, since $\boxx$ is only employed to maximize EI as acquisition function, which is positive, we propose to define $EI_{\text{ext}}: \R^d \to \R$ using a penalization as follows:
$$ EI_{\text{ext}}(\y) = \left\{ 
\begin{array}{l}
EI(\gamma(\y)) \text{~if~} \y \in \Z \\
 -\| \y \| \text{~else}
\end{array}
\right.
$$
where testing if $\y \in \Z$ is performed by checking if the linear system $\B\x = \y$ has a solution $\x \in \X$.
The same test is to be used to build an initial design of experiments in $\Z$.
The penalty $-\| \y \|$ if $\y \notin \Z$ has been chosen to push toward the center of domain, thus toward $\Z$.
An outline of the resulting REMBO procedure with the proposed improvements is given in Algorithm \ref{alg:rembo}. 

\begin{algorithm}[htpb]
\caption{Pseudo code of the REMBO procedure with mapping $\gamma$}
\label{alg:rembo}
\begin{algorithmic}[1]
\Require $d$, $n_0$, kernel $k$ (e.g., among $k_\Y, k_\X, k_\Psi$).
\State Sample $\A \in \mathbb{R}^{D \times d}$ with independent standard Gaussian coefficients.
\State Apply Gram-Schmidt orthonormalization to $\A$.
\State Define $\B = \A^\top$ and compute $\boxx$.
\State Construct an initial design of experiment in $\Z$, of size $n_0$. 
\State Build the GP model with kernel $k$.
\While{time/evaluation budget not exhausted}
\State Find $\y_{n+1} = \argmax_{\y \in \boxx} EI_\text{ext}(\y)$
\State Evaluate the objective function at $\y_{n+1}$, $f_{n+1} = f(\gamma(\y_{n+1}))$.
\State Update the GP model based on new data.
\EndWhile
\end{algorithmic}
\end{algorithm} 

\subsection{Numerical experiments}

We propose to illustrate the interest of the proposed modifications on several
benchmark functions of various dimensionality, as summarized in Table
\ref{tab:funcs}. Some are classical multimodal synthetic functions such as
Branin, Hartman6, Giunta and Levy, which are reasonably well modeled by
Gaussian process. The Borehole function \cite{Morris1993} models the water-flow in a borehole, commonly used in the computer experiments literature. The
last one, the Cola function is a  weighted least squares scaling problem, used
for instance in multidimensional scaling, see, e.g., \cite{Mardia1979}. The
number of influential dimensions of those problems varies from 2 to 17, for a
total number of variables $D$ ranging from 17 to 200. The former are chosen
randomly, but kept fixed for each specific run (25 in total) to ensure
fairness among comparators. The budget for optimization is either 100 or 250
evaluations, which are representative of expensive optimization tasks.\\

\begin{table}[htpb]
\centering
\caption{Summary of test functions}
\begin{tabular}{l|llll||l|lll}
name   & $d$ & $d_e$ & $D$ &  & name  & $d$ & $d_e$ & $D$ \\ \hline
Branin \cite{Dixon1978} & 2   & 2     & 25, 100  &  & Hartman6 \cite{Dixon1978} & 6   & 6     & 50, 200  \\
Giunta \cite{Mishra2006} & 2   & 2    & 80  &  & Borehole \cite{Morris1993} & 8   & 8     & 50  \\
Cola \cite{Mathar1994}  & 6   & 17    & 17  &  & Levy \cite{Laguna2005}    & 10  & 10    & 80 
\end{tabular}
\label{tab:funcs}
\end{table}

The emphasis is on both the average and worst case performances, as with the
new formulation the search space $\Z$ maps with $\gamma$ to the entire $\Emb$.
We compare it to the original choice of REMBO: mapping $\phi$ with search
domain $\Y = [-\sqrt{d}, \sqrt{d}]^d$. Preliminary tests with search domains
encompassing estimated $\U$ showed a degraded performance compared to those
with $\Y$ and are not reported here. We also test the three possible
covariance kernels (see Table \ref{tab:kernels}) with both mappings.
Experiments have been performed relying on the \texttt{DiceKriging} and
\texttt{DiceOptim} packages \cite{Roustant2012}, with an unknown constant
trend and Mat\'ern covariance kernel with $\nu = 5/2$. The corresponding code
is publicly available at \url{https://github.com/mbinois/RRembo}. For solving
the quadratic programming problem within $\gamma$, we use the
\texttt{quadprog} package \cite{Berwin2013}. In all the problems here, the
corresponding extra cost was not more than a dozen of milliseconds per
solve.\\

The baseline performance is given by uniform sampling (RO) in the original
domain $\X$. We also compare to the Ensemble Bayesian Optimization method
\cite{Wang2018} using the \texttt{Python} code made available by the authors.
It relies on an ensemble of additive GP models on randomized partitions of the
space, which has been shown to scale both in terms of dimension and number of
variables. We use the default tuning parameters, with batch size twenty, and let
the number of additive components increase up to $d$.\\

We use the optimality gap as a metric, i.e., best value found minus known
optimum value. The results are provided in Figs. \ref{fig:res_prog} and
\ref{fig:res_fin}, corresponding to final boxplots and progress over
iterations, respectively.  Overall, the median performance of REMBO variants
is better than both uniform sampling (RO) and ensemble Bayesian optimization
(EBO). Moreover, the worst performance in terms of $75\%$ quantile is almost
always improved with the mapping $\gamma$.  Between the three kernel choices,
$k_\Psi$ is consistently a sound choice, while the performance of $k_\X$ is
highly variable. As a result, looking at the best rank over all tests,
$\gamma$ with $k_\Psi$ is the best combination.\\

The results are the most mitigated for the $d = 2$ cases, where the mapping
$\phi$ can outperform the mapping $\gamma$. As $d$ increases, the difference
becomes more striking in favor of $\gamma$, and its $75\%$ quartile is always
below the $25\%$ quartile from RO. In the Levy case, where $d = 10$, the
original REMBO method is even worse than RO.  Independently of the kernels, a
proportion of under-performing outliers with mapping $\phi$ and fixed
$\Y=[-\sqrt{d}, \sqrt{d}]^d$ can be interpreted as cases when the optimal
solution is not contained in the domain; these do not happen with $\gamma$.
For some of the remaining ones, the reason may be related to an unfavorable
arrangements of strips for the GP modeling that could be alleviated with
further work on kernels.\\

Figure \ref{fig:res_prog} also shows differences in initial designs with
respect to the mapping used. There is no clear trend since the best design
strategy depends on the problem at hand and the location of the optima. For
instance, in the case of Borehole, designs using $\phi$ on $\Y = [-\sqrt{d},
\sqrt{d}]^d$ are better starting points than with $\gamma$ on $\Z$, but this
advantage is quickly reduced. On a different aspect, shown on the Branin and
Hartman6 functions, increasing $D$ does not affect the performance of the
REMBO methods more than the choice of the active dimensions. We did not
conduct such a study for EBO. The Cola function illustrates the case when all
$D$ variables are influential. Even if this is not a favorable setup for
REMBO, it still outperforms RO and EBO with limited budgets.\\

Under the limited budgets used here, relying on random embeddings proved much
better than uniform sampling. The only exception -- for the original method
only -- is with Levy, highlighting that the choice of the domain is crucial with respect to
the performance of the method. It also illustrates that when the budget is
low, it may be detrimental to balance observations to learn the structure of
the function, such as with EBO.
On all examples, considering problem ($\pbR'$) thus appears as a sound
alternative to ($\pbR$). Indeed, initial concerns that a larger search space
may impact the average performance do not reflect on the results, even often
showing a superior performance. As for robustness, the worst performances have
been greatly improved in general.

\begin{figure}%
\centering%
\includegraphics[width=0.45\textwidth, trim = 30 43 20 23, clip = TRUE]{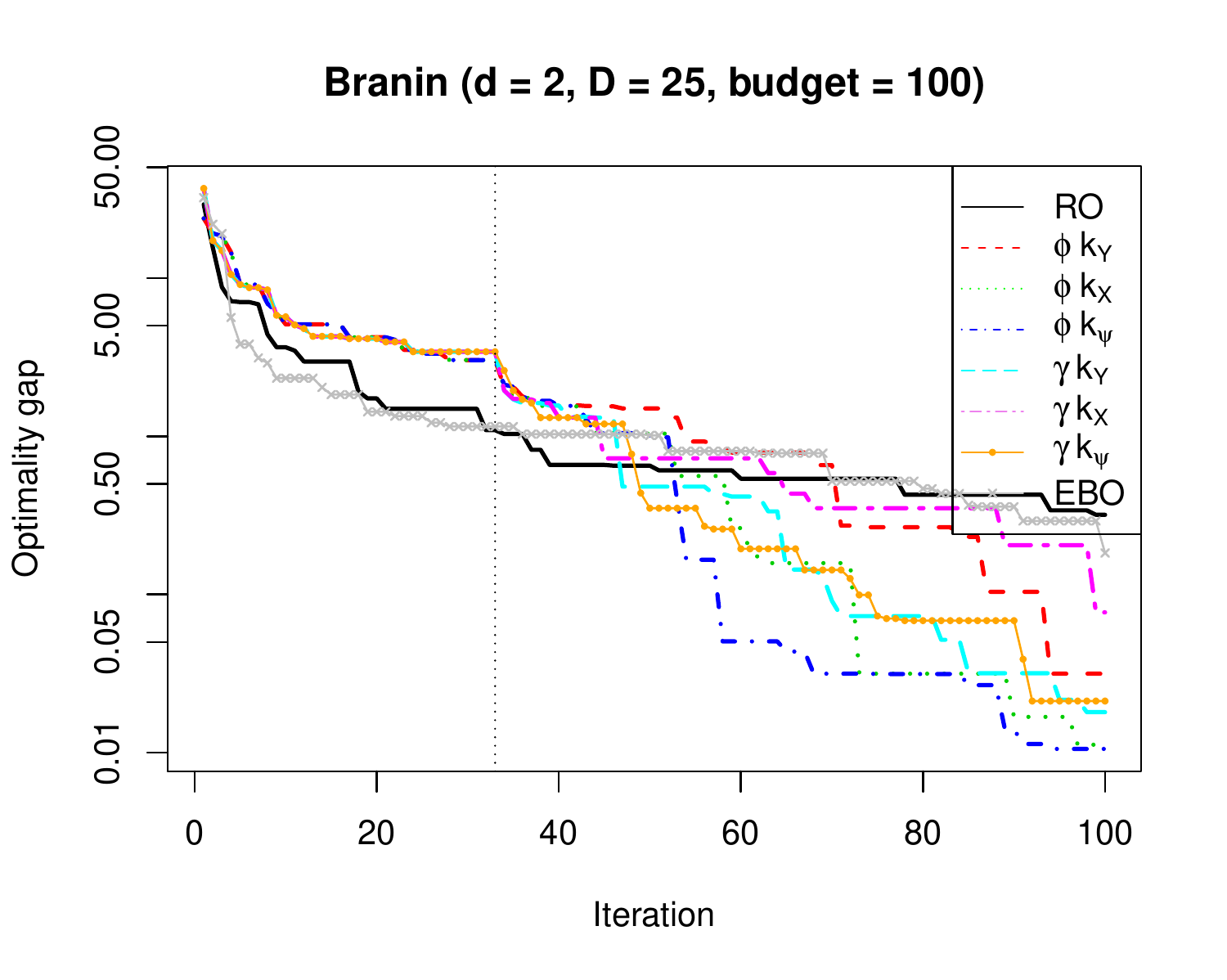}%
\includegraphics[width=0.45\textwidth, trim = 30 43 20 23, clip = TRUE]{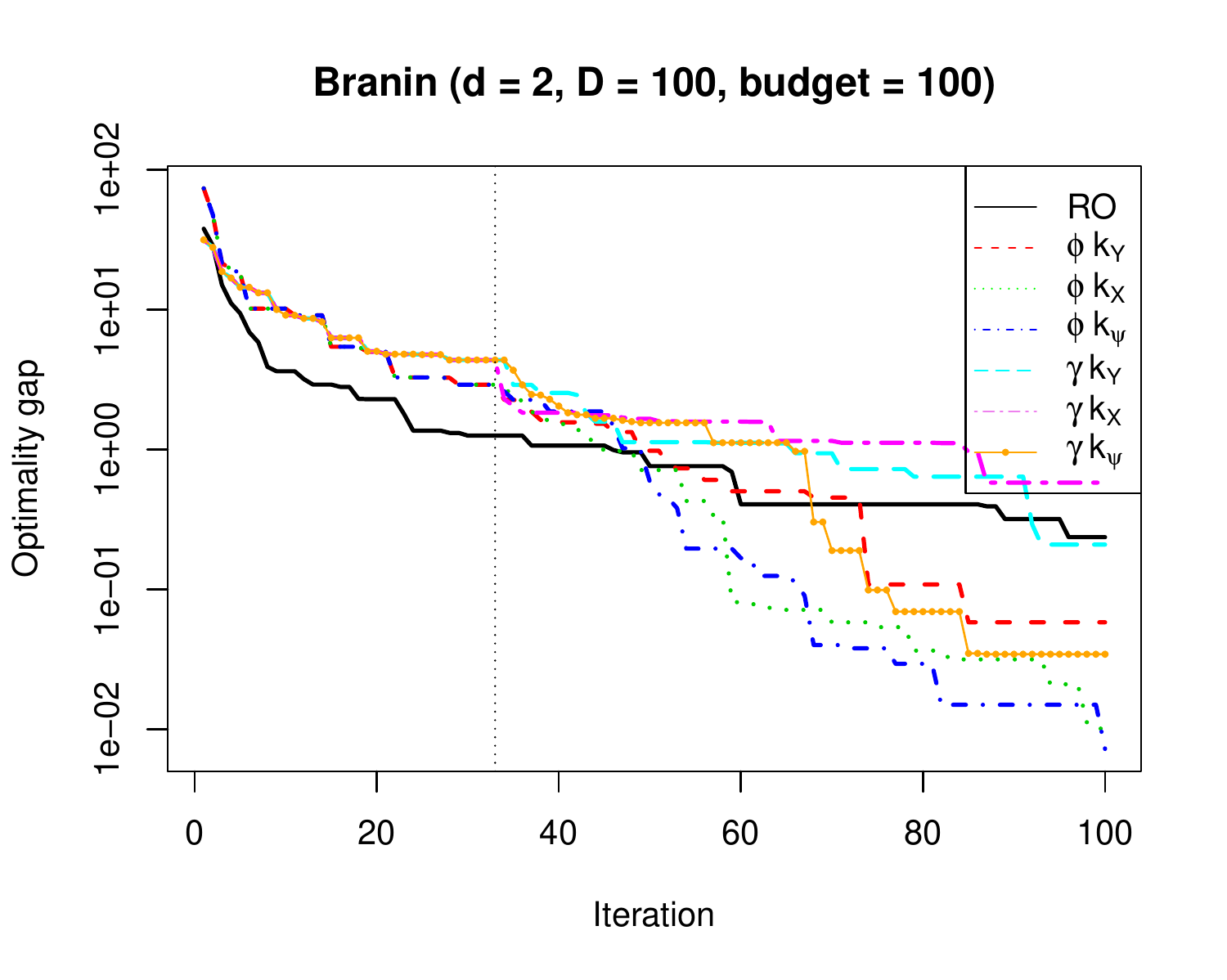}\\
\includegraphics[width=0.45\textwidth, trim = 30 43 20 23, clip = TRUE]{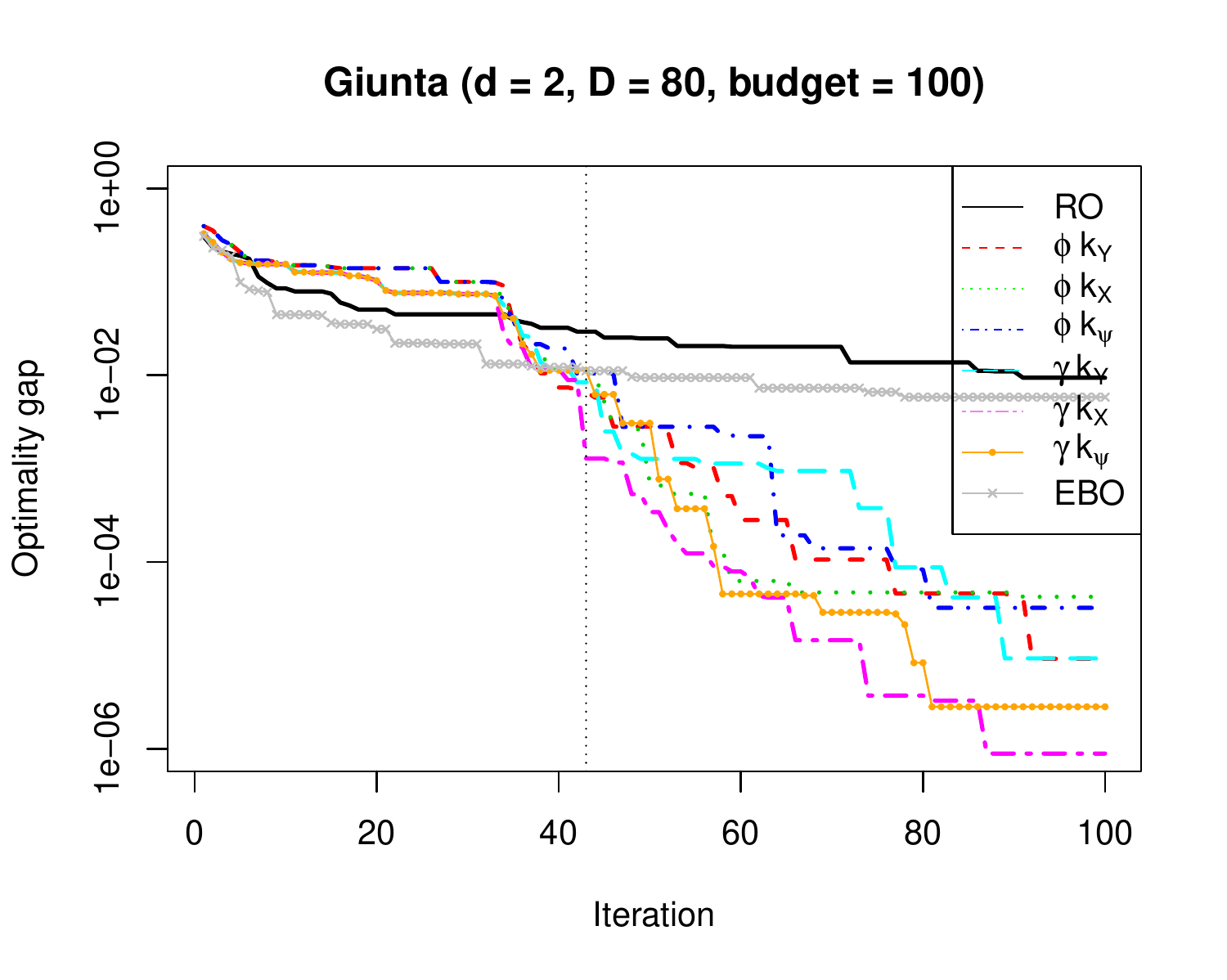}%
\includegraphics[width=0.45\textwidth, trim = 30 43 20 23, clip = TRUE]{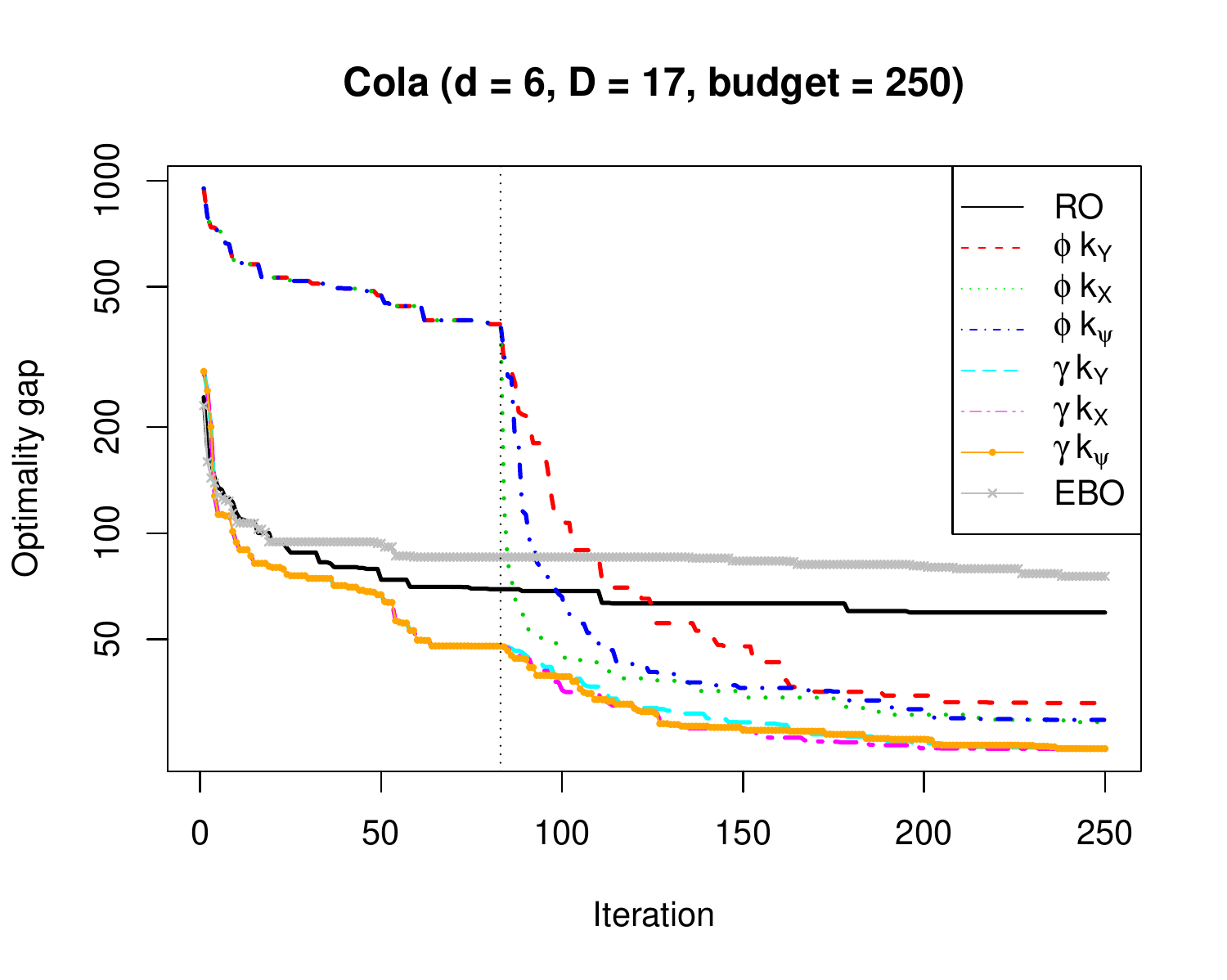}\\
\includegraphics[width=0.45\textwidth, trim = 30 43 20 23, clip = TRUE]{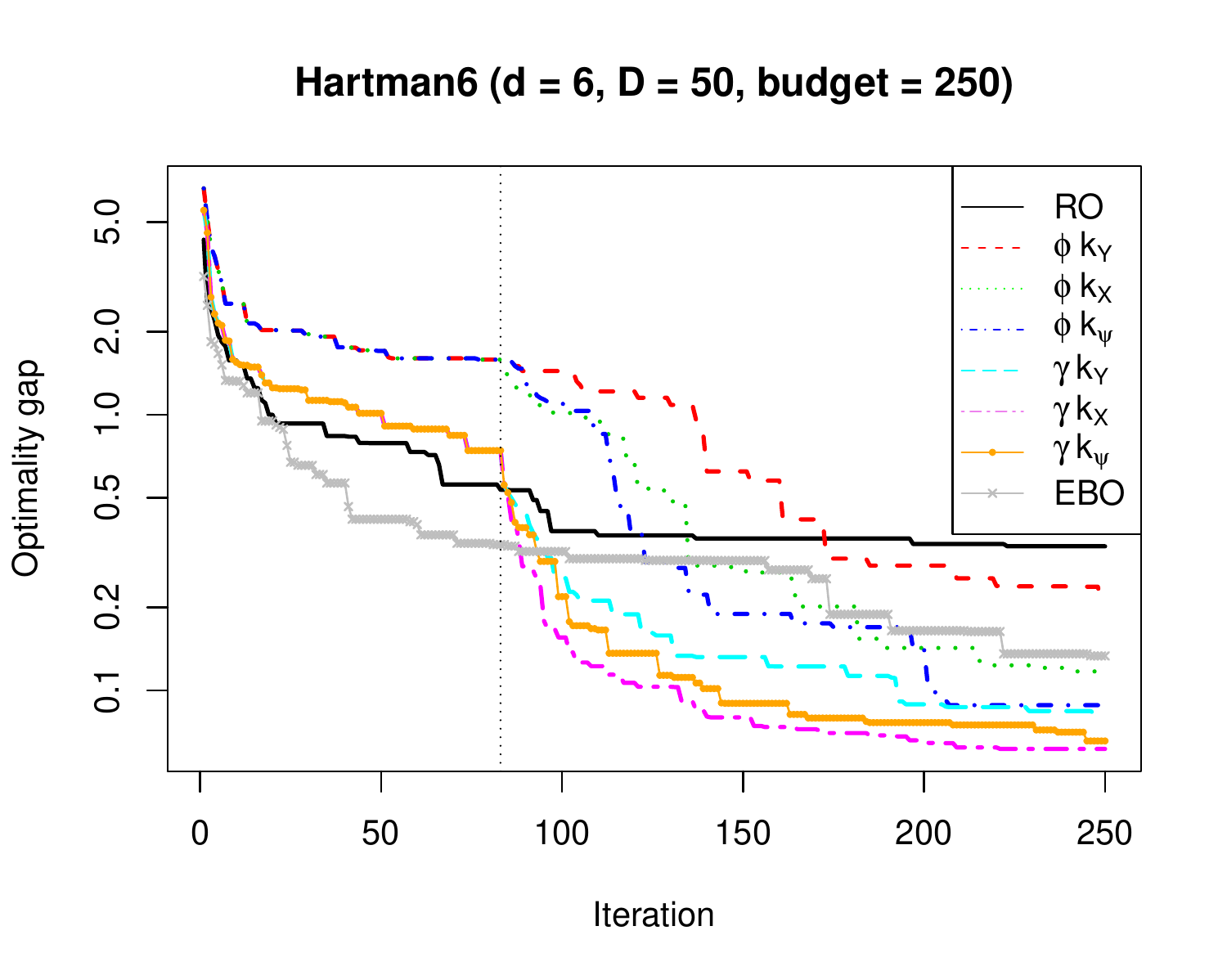}%
\includegraphics[width=0.45\textwidth, trim = 30 43 20 23, clip = TRUE]{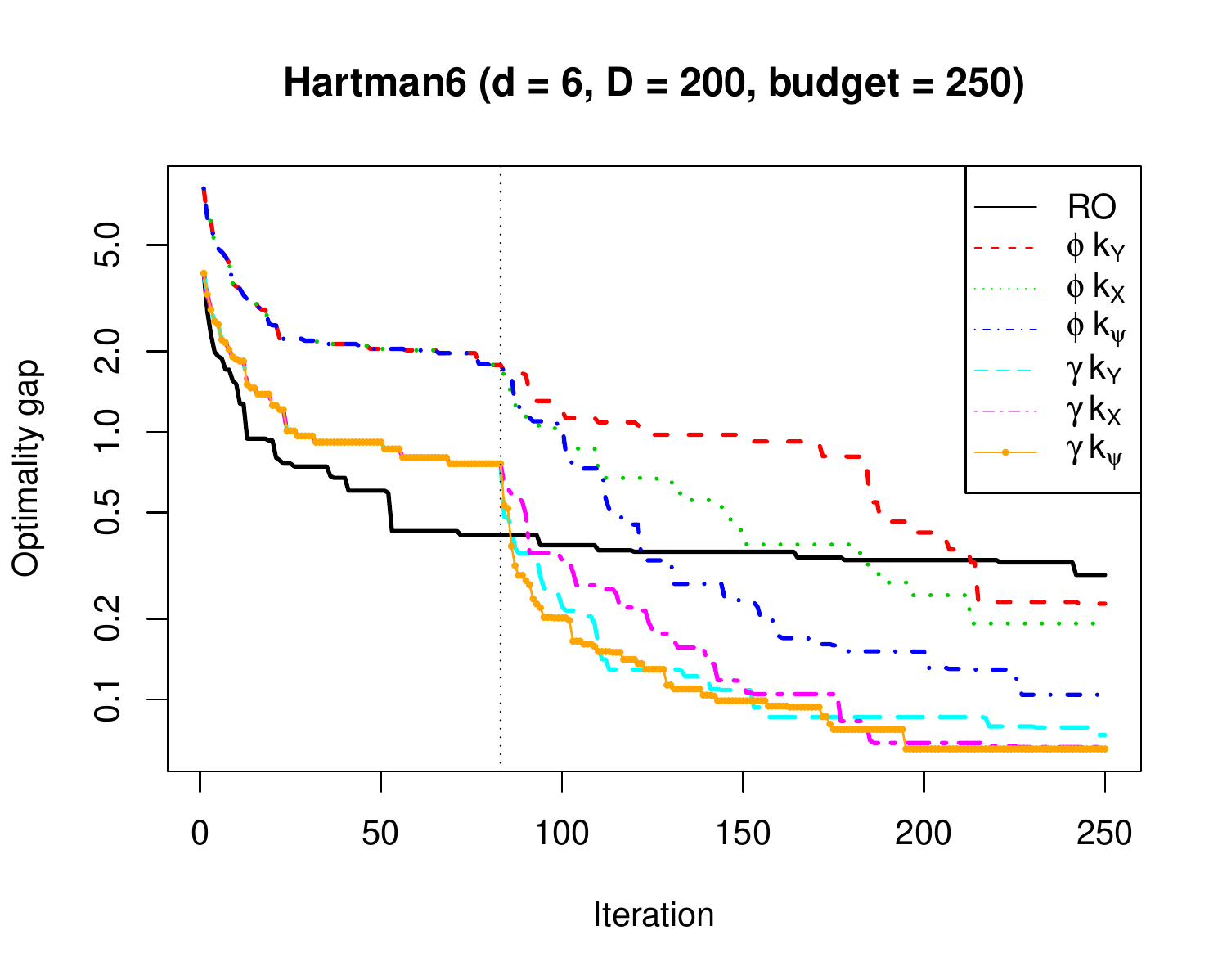}\\
\includegraphics[width=0.45\textwidth, trim = 30 46 20 23, clip = TRUE]{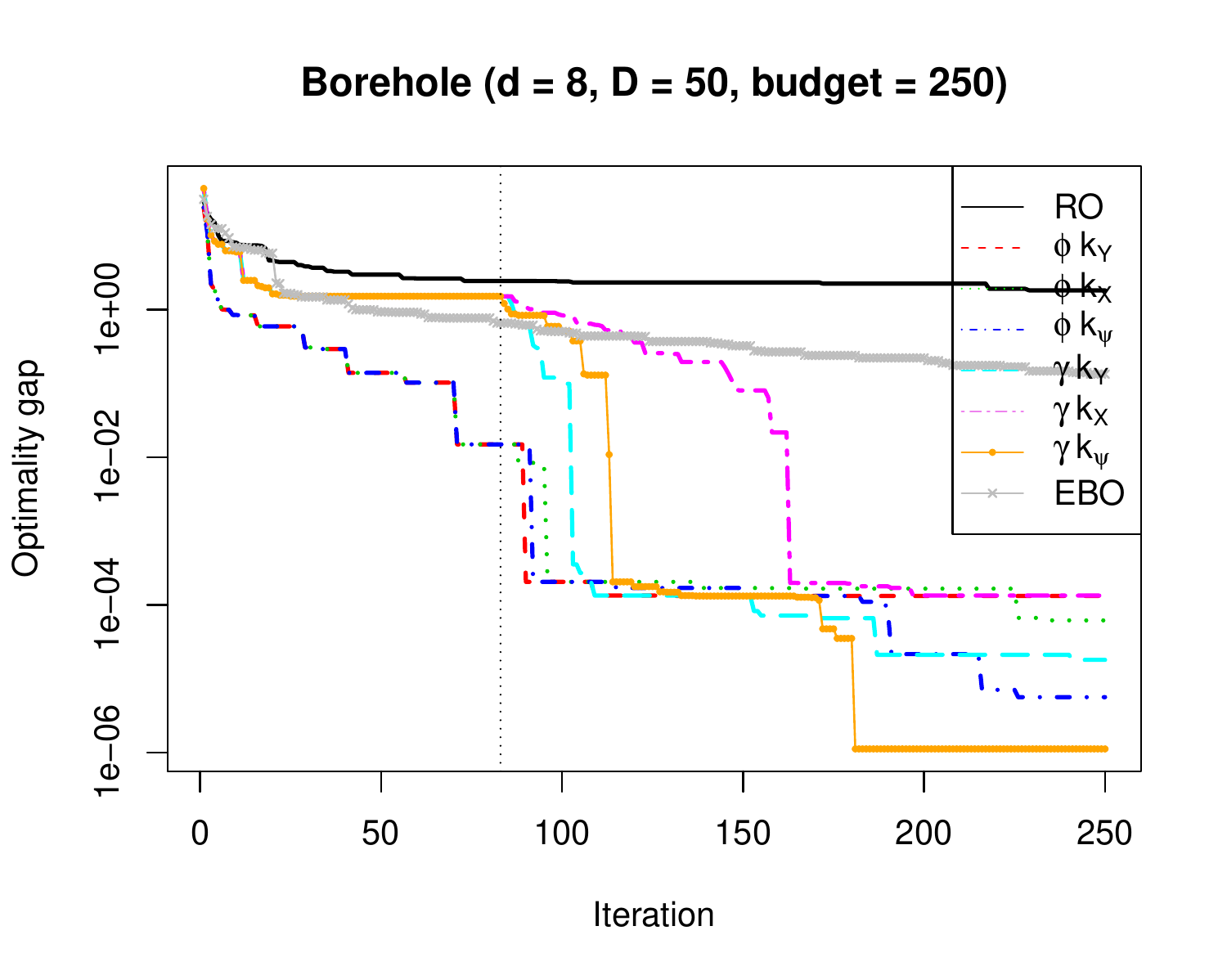}%
\includegraphics[width=0.45\textwidth, trim = 30 46 20 23, clip = TRUE]{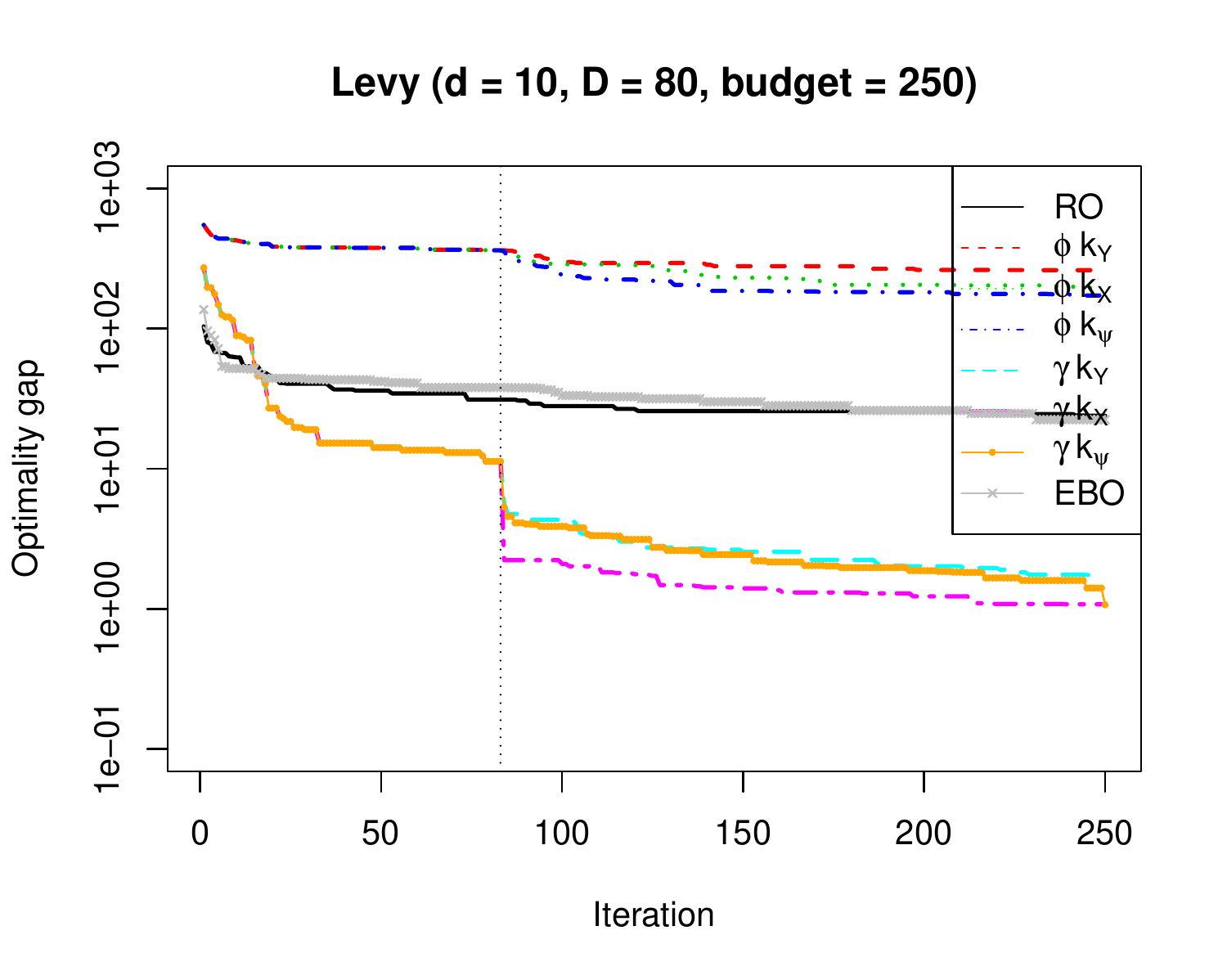}%

\caption{Decrease of median optimality gap (log scale) over iterations for random optimization (RO), ensemble Bayesian optimization (EBO) as well as variants of REMBO with mappings $\phi$ (on $[-\sqrt{d},
\sqrt{d}]^d$) or $\gamma$  (on $\Z$) and kernels in Table \ref{tab:kernels}, on the test problems of Table \ref{tab:funcs}. The dotted vertical lines marks the end of the design of experiments phase.}
\label{fig:res_prog}
\end{figure}
\begin{figure}%
\centering%
\includegraphics[width=0.45\textwidth, trim = 30 40 20 23, clip = TRUE]{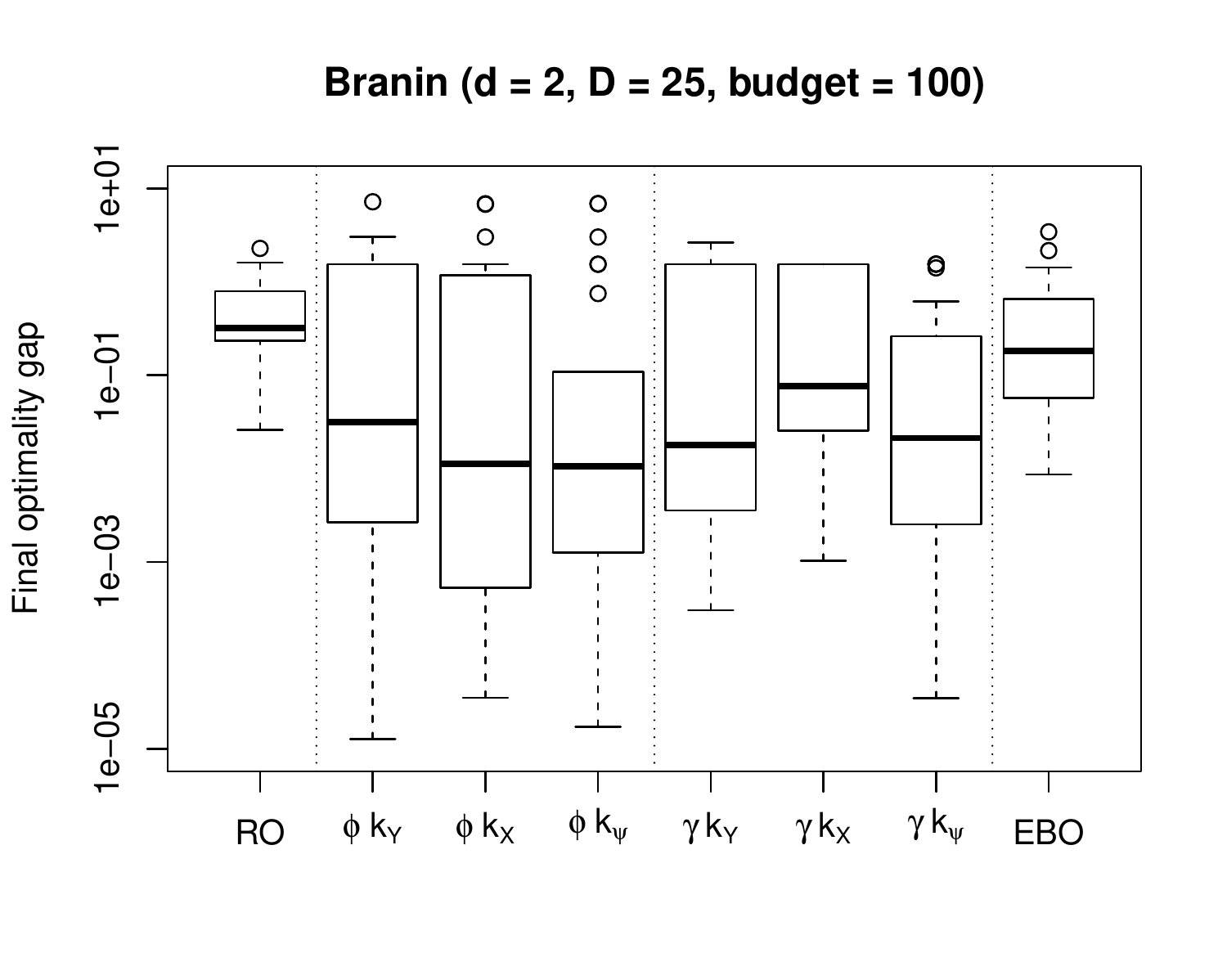}%
\includegraphics[width=0.45\textwidth, trim = 30 40 20 23, clip = TRUE]{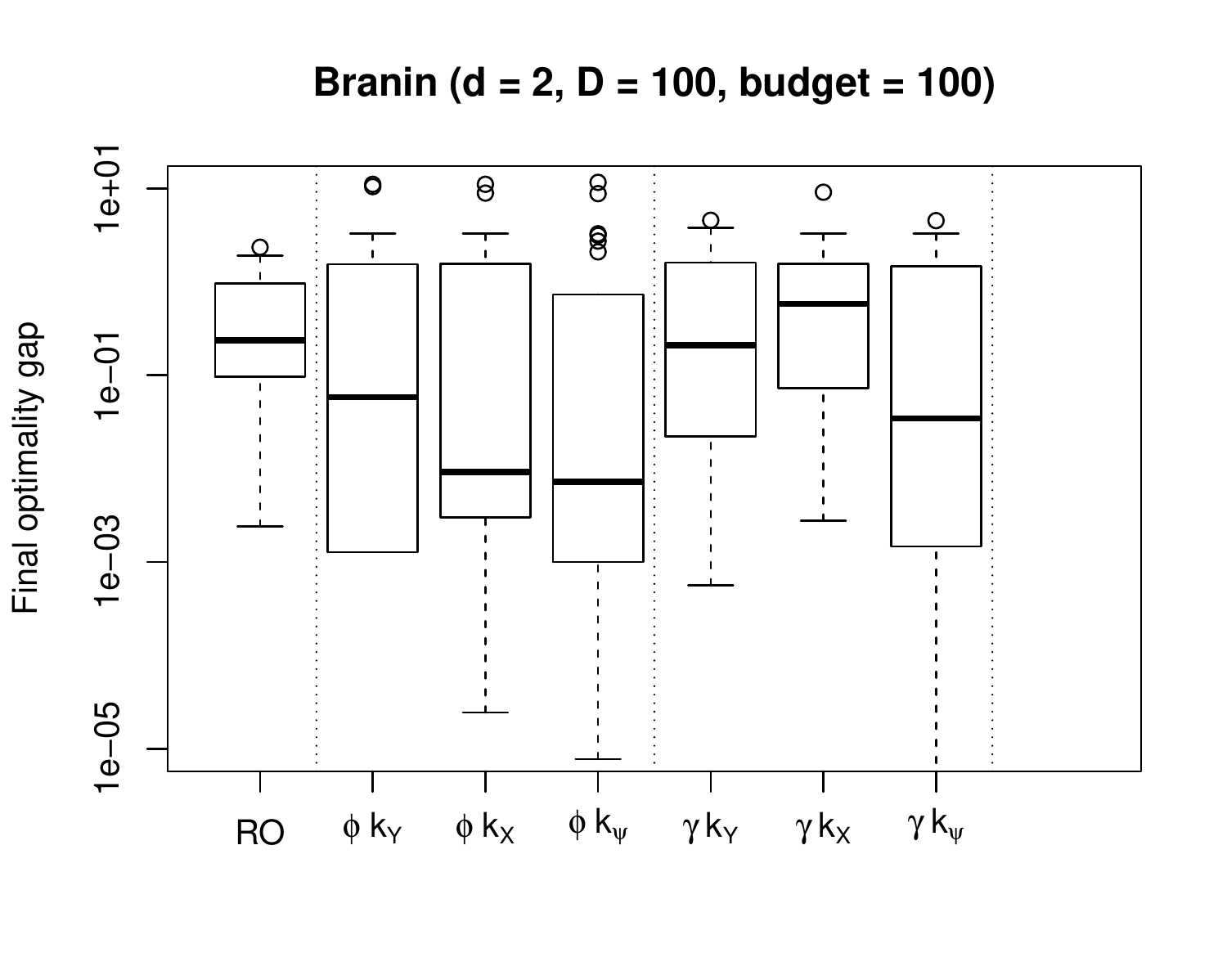}\\
\includegraphics[width=0.45\textwidth, trim = 30 40 20 23, clip = TRUE]{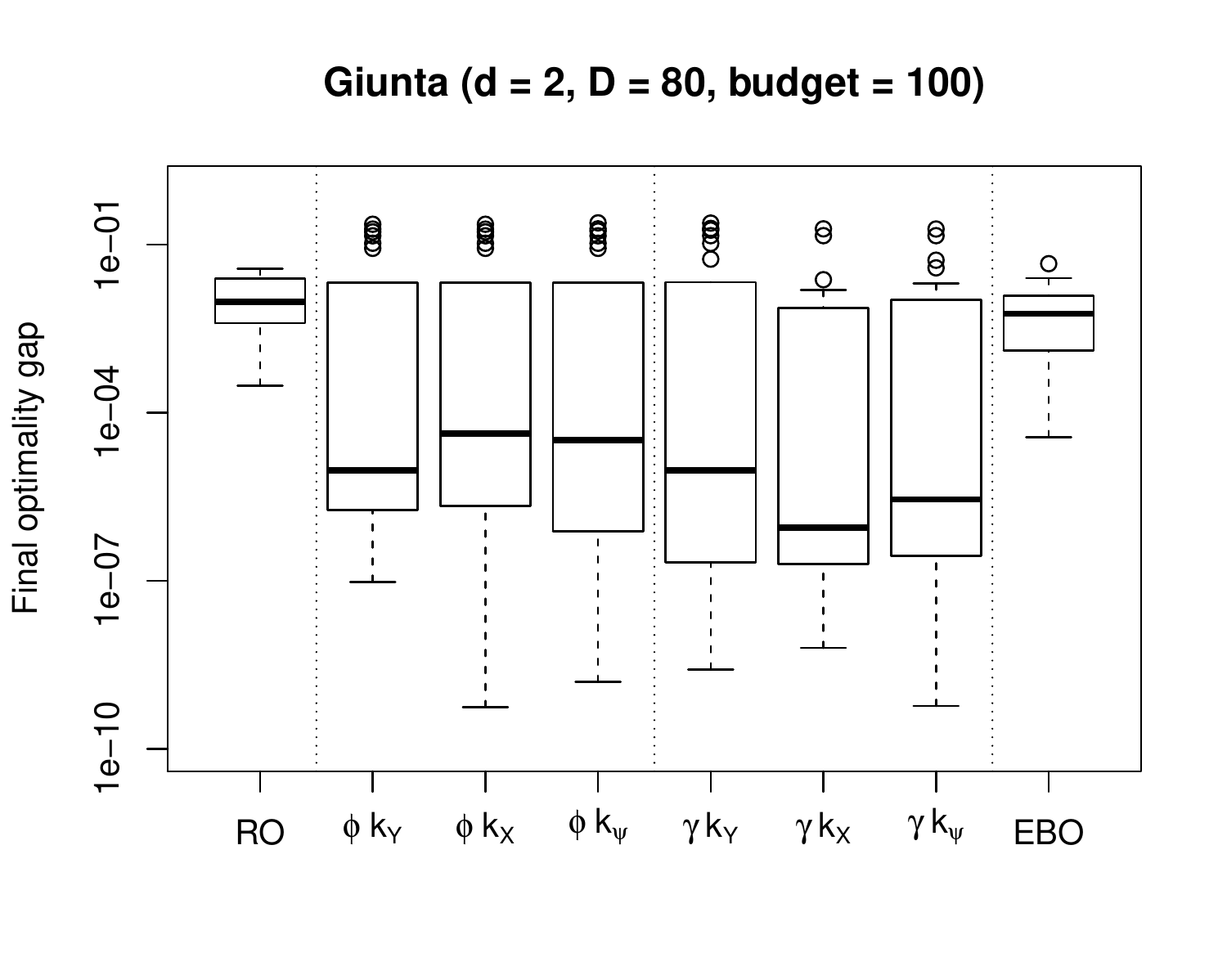}%
\includegraphics[width=0.45\textwidth, trim = 30 40 20 23, clip = TRUE]{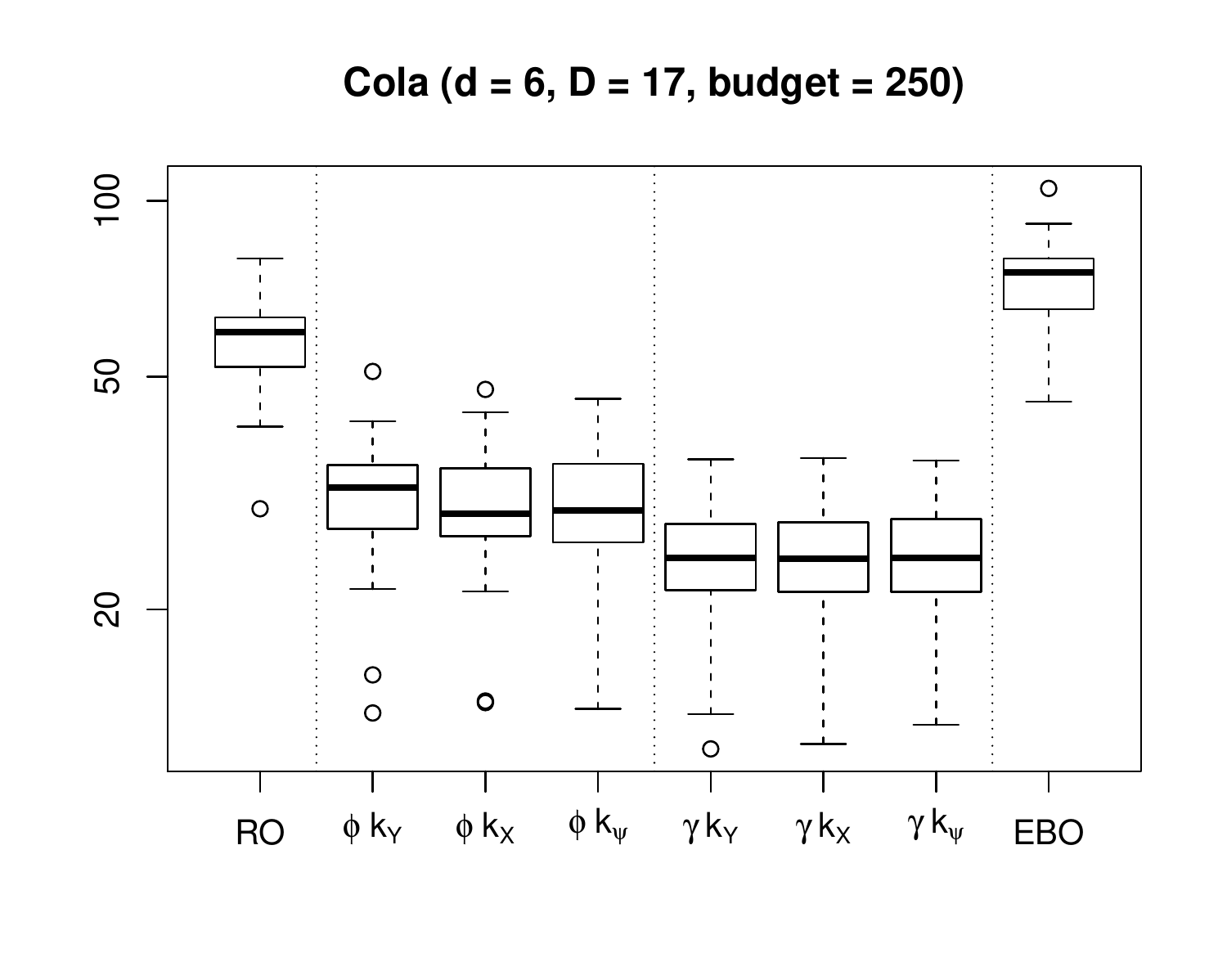}\\
\includegraphics[width=0.45\textwidth, trim = 30 40 20 23, clip = TRUE]{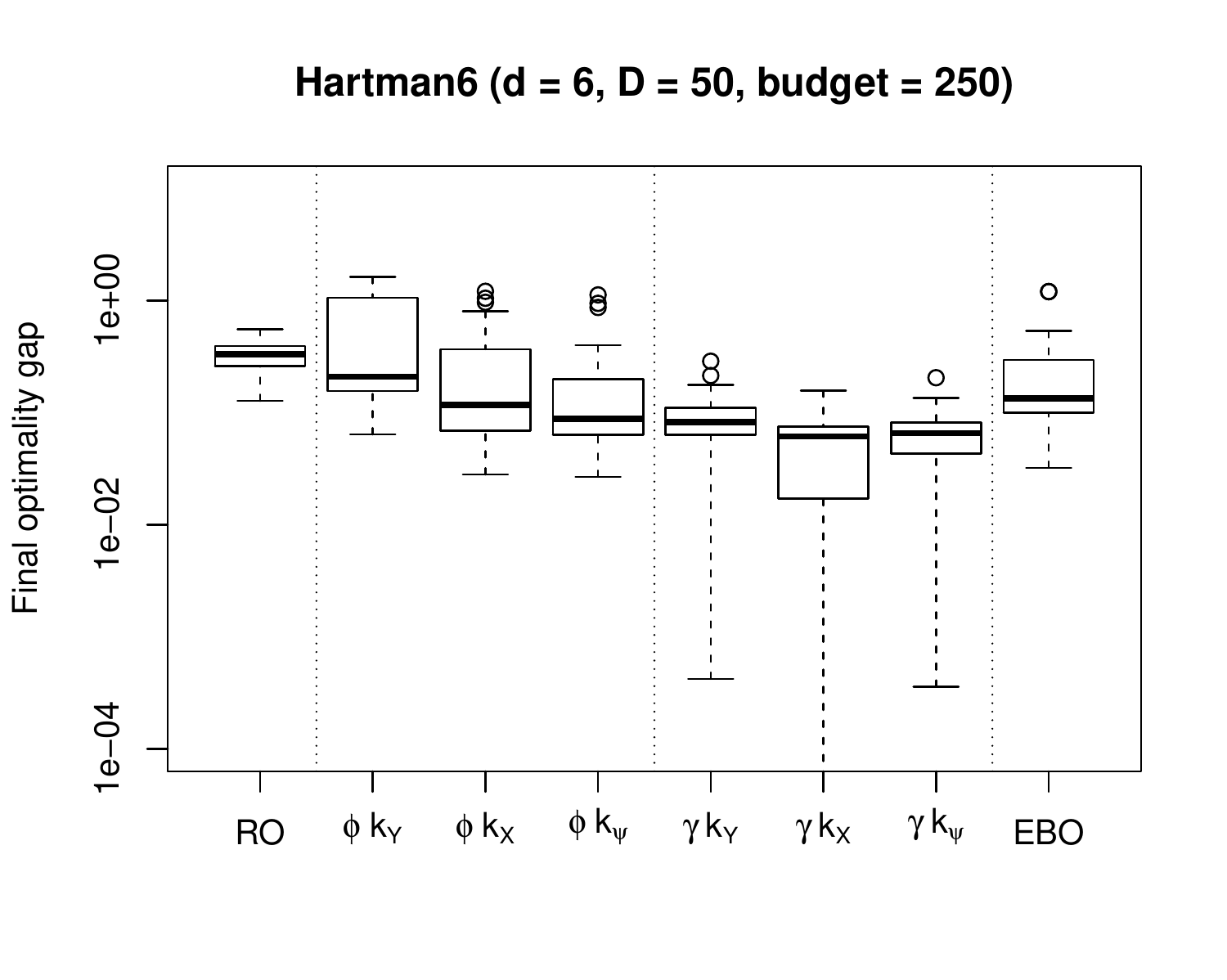}%
\includegraphics[width=0.45\textwidth, trim = 30 40 20 23, clip = TRUE]{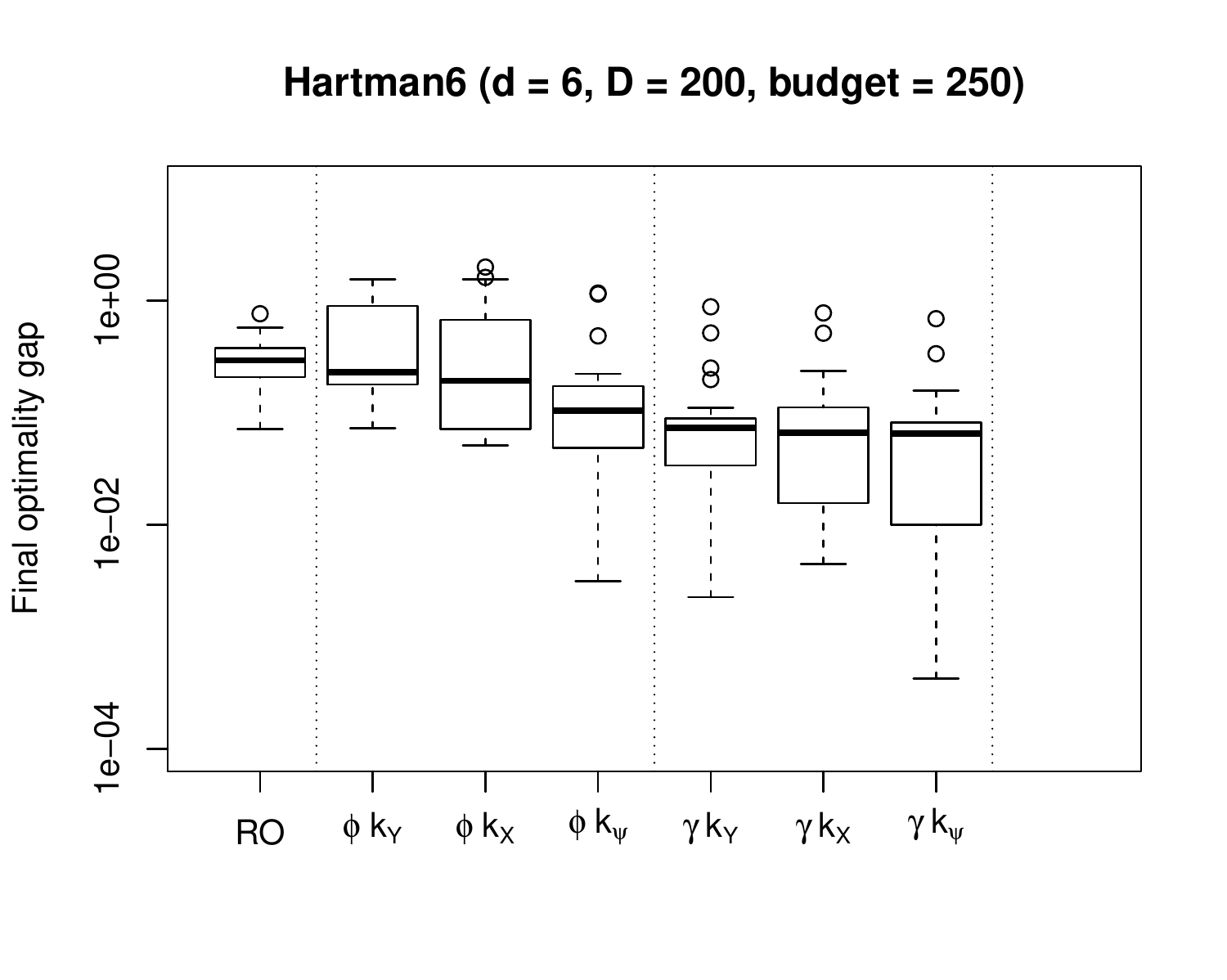}\\
\includegraphics[width=0.45\textwidth, trim = 30 40 20 23, clip = TRUE]{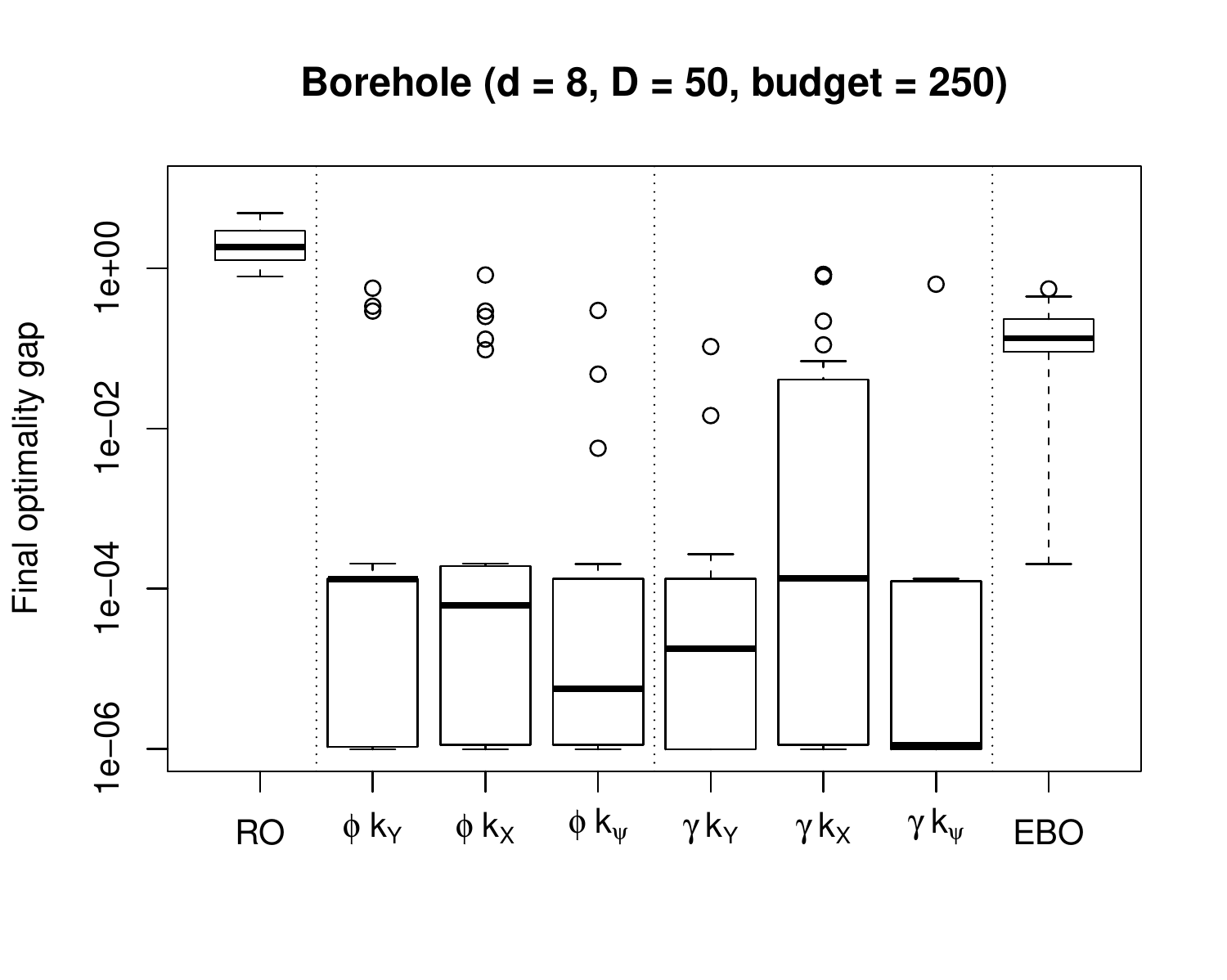}%
\includegraphics[width=0.45\textwidth, trim = 30 40 20 23, clip = TRUE]{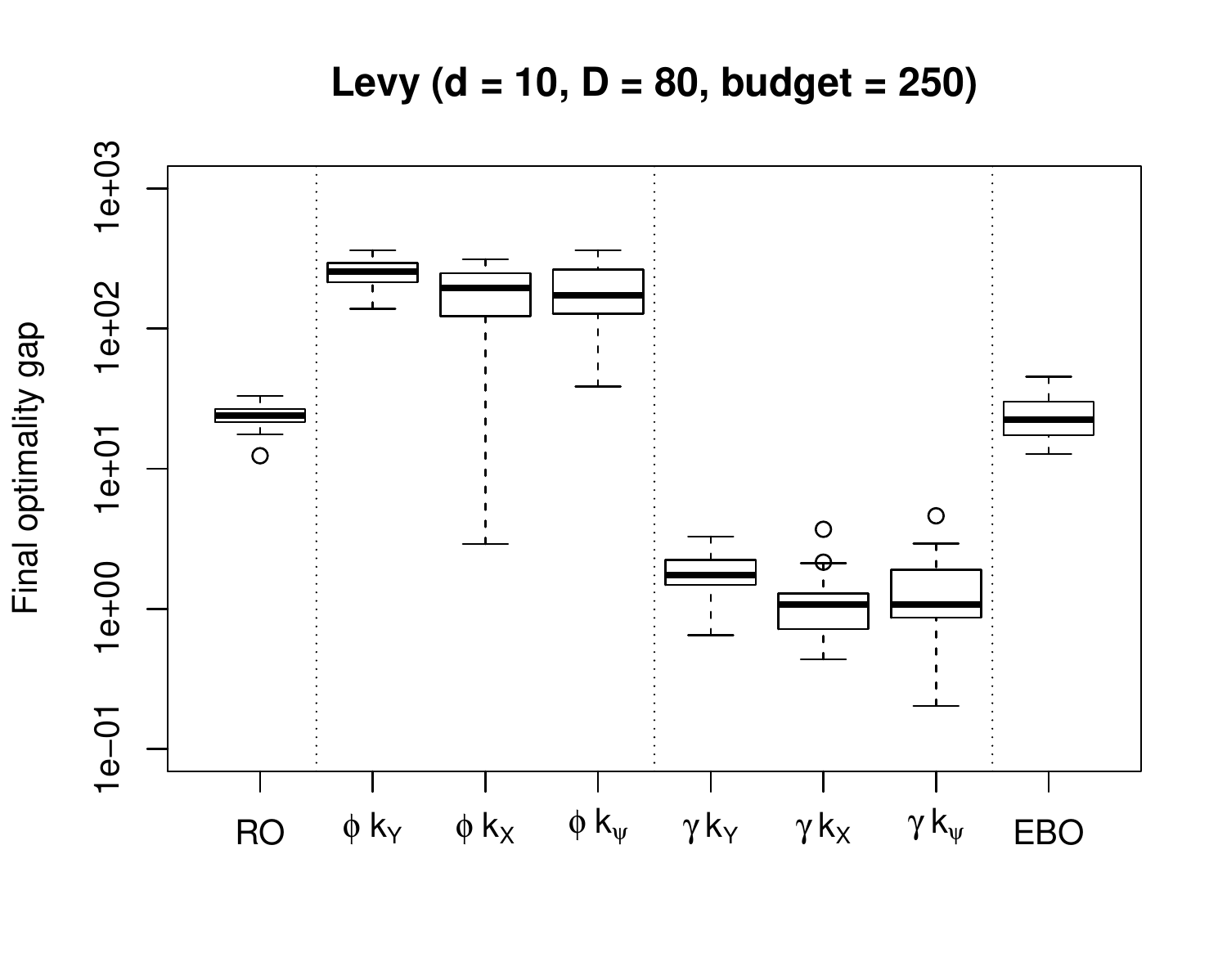}%

\caption{Boxplots of optimality gap (log scale), corresponding to the last iteration in Figure \ref{fig:res_prog}, grouped by mappings.}
\label{fig:res_fin}
\end{figure}

\section{Conclusion and perspectives}
\label{sec:ccl}

Although random embeddings offers a simple yet powerful framework to perform high-dimensional optimization, 
a snag relies in the definition of bounds for the low-dimensional search space. 
In the original setting, this results in an unsatisfactory compromise between hindering efficiency or taking the risk of discarding global solutions.
While for this latter guarantees were given with probability one only for the entire low-dimensional search space $\R^d$, 
we show that it is sufficient to take specific compact sets. 
Our main outcome is to explicitly describe these minimal sets for searching a solution under the random embedding paradigm.\\ 

By pointing out to the difficulties that originate from the convex projection,
we propose to alleviate these drawback by amending this component.
In particular, we show that using an alternative embedding procedure yields a more convenient minimal set to work with, 
that is, relying on a back-projection from the orthogonal projection of the high-dimensional search space. 
We further show on examples that, in this case, the gain in robustness of discarding the risk of missing the optimum on the embedded low-dimensional space outweighs the increase in size of the search space.\\ 

The benefits could be even greater when extending the random embedding technique to constrained or multi-objective optimization, as tested e.g., in \cite{Binois2015b}. 
Indeed, the impact of restricting the search space too much could be even more important. 
Concerning Bayesian optimization, in addition to consider batch-sequential and other acquisition functions, 
perspectives include investigating non-stationary models to further improve the GP modeling aspect, based on the various properties uncovered.
Finally, a promising approach would be to hybridize REMBO with methods that learn the low-dimensional structure as in \cite{Garnett2014}. 

\section*{Acknowledgments}
We thank the anonymous reviewers for helpful comments on the earlier version of the paper.
Parts of this work have been conducted within the frame of the ReDice Consortium,
gathering industrial (CEA, EDF, IFPEN, IRSN, Renault) and academic (Ecole des Mines de Saint-Etienne, INRIA, and the University of Bern) partners around advanced methods for Computer Experiments.
M. B. also acknowledges partial support from National Science Foundation grant DMS-1521702.

\bibliographystyle{spmpsci}
% \bibliography{biblio}

\appendix

\section{Proofs}

\subsection{Properties of the convex projection}

We begin with two elementary properties of the convex projection onto the hypercube $\X = [-1,1]^D$:

\begin{Property}[Tensorization property]
\label{property:P1} 
$\forall \x \in \R^D$, $p_\X 
  \begin{pmatrix}
  x_1 \\ \dots \\ x_D
  \end{pmatrix}
  = 
  \begin{pmatrix}
  p_{[-1,1]}(x_1) \\ \dots \\ p_{[-1,1]}(x_D)
  \end{pmatrix}
  .$
\end{Property}

\begin{Property}[Commutativity with some isometries]
\label{property:P2} 
Let $q$ be an isometry represented by a diagonal matrix with terms $\varepsilon_i = \pm 1$, $1 \leq i \leq D$. Then, for all $\x \in \R^D$,  $p_\X(q(\x)) = 
  \begin{pmatrix}
  \varepsilon_1 p_{[-1,1]}(x_1) \\ \dots \\ \varepsilon_D p_{[-1,1]}(x_D)
  \end{pmatrix}
  = q(p_\X(\x))$.
\end{Property}

\subsection{Proof of Theorem \ref{prop:ZH}}
\label{ap:proofth1}
\begin{proof}
First, note that $\U$ is a closed set as a finite union of closed sets. Then, let us show that $p_\X(\A \U) = \Emb$. Consider $\x \in \Emb$, hence $|x_i| \leq 1$ and $\exists \y \in \R^d$ s.t. $\x = p_\X(\A \y)$. Denote $\vecb = \A \y$.  
We distinguish two cases:
\begin{enumerate}
  \item More than $d$ components of $\vecb$ are in $[-1,1]$. Then there exists a set $I \subset \{1, \dots, D \}$ of cardinality $d$ such that $\y \in \bigcap \limits_{i \in I} \Band_i = \Paral_I \subseteq \U$, implying that $\x \in p_\X(\A \U)$. 
\item $0 \leq k < d$ components of $\vecb$ are in $[-1,1]$. 
It is enough to consider that $\vecb \in [0, +\infty)^D$. 
Indeed, for any $\x \in \Emb$, any $\A \in \classA$, let $\boldsymbol{\varepsilon}$ be the isometry given by the diagonal $D \times D$ matrix $\boldsymbol{\varepsilon}$ with elements $\pm 1$ such that $\boldsymbol{\varepsilon} \vecx \in [0, +\infty)^D$. It follows that $\boldsymbol{\varepsilon}\vecb$ is in $[0, +\infty)^D$ too. 
Denote $\x' = \boldsymbol{\varepsilon} \x$, $\vecb' = \boldsymbol{\varepsilon} \vecb$ and $\A' = \boldsymbol{\varepsilon} \A$.  
Thus if $\exists \vecu \in \U$ such that $\x' = p_\X(\vecb') = p_\X(\A' \vecu)$, by property \ref{property:P2}: $\boldsymbol{\varepsilon} \x = \boldsymbol{\varepsilon} p_\X(\A \vecu)$ leading to $\x = p_\X(\vecb) = p_\X(\A \vecu)$. 
From now on, we therefore assume that $b_i \geq 0$, $1 \leq i \leq D$.
Furthermore, we can assume that $0 \leq b_1 \leq \dots \leq b_D$, from a permutation of indices. Hence $b_i > 1$ if $i > k$ and  $\x = (x_1 = b_1, \dots, x_k = b_k, 1, \dots, 1)^T$.\\

Let $\y' \in \R^d$ be the solution of $\A_{1, \dots,d} \y' = (b_1, \dots b_k, 1, \dots, 1)^T$ (vector of size $d$). Such a solution exists since $\A_{1, \dots,d}$ is invertible by hypothesis. Then define $\vecb' = \A \y'$, $\vecb' = (b_1, \dots, b_k,1, \dots, 1, b_{d+1}', \dots, b_D')^T$. We have $\vecb' \in \Ran(\A)$ and $\y' \in \Paral_{1, \dots, d} \subseteq \U$.
\begin{itemize}
  \item If $\min_{i \in \{d+1, \dots, D\}}(b_i') \geq 1$, then $p_\X(\vecb') = p_\X(\vecb) = \x$, and thus $\x = p_\X(\A \y') \in p_\X(\A \U)$.
\item Else, the set $L = \{i \in \N: d+1 \leq i \leq D \telque b'_i < 1\}$ is not empty. Consider $\vecc = \lambda \vecb' + (1-\lambda)\vecb$, $\lambda \in ]0,1[$. By linearity, since both $\vecb$ and $\vecb'$ belong to $\Ran(\A)$, $\vecc \in \Ran(\A)$.  
\begin{itemize}
  \item For $1 \leq i \leq k$, $c_i = x_i$.
  \item For $k+1 \leq i \leq d$, $c_i = \lambda + (1- \lambda)b_i \geq 1$ since $b_i > 1$.
  \item For $i \in \{d+1, \dots, D\} \setminus L$, $b'_i \geq 1$ and $b_i > 1$ hence $c_i \geq 1$.
  \item We now focus on the remaining components in $L$. For all $i \in L$, we solve $c_i = 1$, i.e., $\lambda b'_i + (1-\lambda) b_i = \lambda (b'_i - b_i) + b_i = 1$. The solution is $\lambda_i = \frac{b_i-1}{b_i - b'_i}$, with $b_i - b'_i > 0$ since $b'_i < 1$. Also $b_i - 1 > 0$ and $b_i - 1 < b_i - b'_i$ such that we have $\lambda_i \in ]0,1[$. Denote $\lambda^* = \min_{i \in L} \lambda_i$ and the corresponding index $i^*$. By construction, $c_{i^*} = 1$ and $\forall i \in L$, $c_i = \lambda^* (b'_i - b_i) + b_i \geq \lambda_i (b'_i - b_i) + b_i = 1$ since $\lambda_i \geq \lambda^*$ and $b'_i - b_i < 0$.
\end{itemize}

To summarize, we can construct $\vecc^*$ with $\lambda^*$ that has $k + 1$ components in $[-1,1]$ (the first $k$ and the $i^{*th}$ ones), the others are greater or equal than 1. Moreover, $\vecc^* \in \Ran(\A)$ and fulfills $p_\X(\vecc^*) = p_\X(\vecb) = \x$ by Property \ref{property:P1}. If $k+1 = d$, this corresponds to case 1 above, otherwise, it is possible to reiterate by taking $\vecb = \vecc$. Hence we have a pre-image of $\x$ by $\phi$ in $\U$.   

\end{itemize}   
\end{enumerate}
Thus the surjection property is shown. There remains to show that $\U$ is the smallest closed set achieving this, along with additional topological properties.\\

Let us show that any closed set $\Y \in \R^d$ such that $p_\X(\A \Y) = \Emb$ contains $\U$. To this end, we consider $\U^* = \bigcup \limits_{I \subseteq \{1, \dots, D\}, |I| = d} \mathring{\Paral}_I$ with $\mathring{\Paral}_I = \set{\y \in \R^d \telque \forall i \in I, -1 < \A_i \y < 1}$, the interior of the parallelotopes. We have $\restriction{\phi}{\mathring{U}}$ bijective. Indeed, all $\x \in p_\X(\A \U^*)$ have (at least) $d$ components whose absolute value is strictly lower that $1$. Without loss of generality, we suppose that they are the $d$ first ones, $I = \{1, \dots, d\}$. Then there exists a \emph{unique} $\y \in \R^d$ s.t. $\x = p_\X(\A \y)$ because $\x_I = (\A \y)_I = \A_I \y$ has a unique solution with $\A_I$ is invertible.  
Since $\Y$ is in surjection with $\Emb$ for $\restriction{\phi}{\Y}$ and $\restriction{\phi}{\U^*}$ is bijective, $\U^* \subseteq \Y$. Additionally, $\Y$ is a closed set so it must contain the closure $\U$ of $\U^*$.\\ 

Finally let us prove the topological properties of $\U$. Recall that parallelotopes $\Paral_I$ $(I \subseteq \{1, \dots, D \})$ are compact convex sets as linear transformations of $d$-cubes.\\ 
Thus $\I = \bigcap\limits_{I \subseteq \{1, \dots, D \}, |I| = d} \Paral_I$ is a compact convex set as the intersection of compact convex sets, which is non empty ($O \in \I$). It follows that $\U = \bigcup\limits_{I \subseteq \{1, \dots, D \}, |I| = d} \Paral_I$ is compact and connected as a finite union of compact connected sets with a non-empty intersection, i.e., $\I$. Additionally $\U$ is star-shaped with respect to any point in $\I$ (since $\I$ belongs to all parallelotopes in $\U$). 
\end{proof}

\subsection{Proof of Proposition \ref{prop:eq_zon}}
\label{ap:proof2}
\begin{proof}

It follows from Definition \ref{def:zon} that $p_\A(\X)$ is a zonotope of center $O$, obtained from the orthogonal projection of the $D$-hypercube $\X$. As such, $p_\A(\X)$ is a convex polytope.

Since $\Emb \subset \X$, it is direct that $p_\A(\Emb) \subseteq p_\A(\X)$.

To prove $p_\A(\X) \subseteq p_\A(\Emb)$, let us start by vertices. 
Denote by $\x \in \R^D$ a vertex of $p_\A(\X)$.\\ 
If $\x \in \X$, then $p_\A(p_\X(\x)) = p_\A(\x) = \x$, i.e., $\x$ has a pre-image in $\Emb$ by $p_\A$.\\
Else, if $\x \notin \X$, consider the vertex $\vecv$ of $\X$ such that $p_\A(\vecv) = \x$. 
Suppose that $\vecv \notin \Emb$. Let us remark that if $\vecv$ is a vertex of $\X$ such that $\vecv \notin \Emb$, then $\Ran(\A) \cap H_v = \emptyset$, where $H_v$ is the open hyper-octant (with strict inequalities) that contains $\vecv$. Indeed, if $\x \in \Ran(\A) \cap H_v$, $\exists k \in \R^*$ such that $p_\X(k \x) = \vecv$, which contradicts $\vecv \notin \Emb$. Denote by $\vecu$ the intersection of the line $(O\x)$ with $\X$, since $\x \notin H_v$, $\vecu \notin H_v$ either, hence $\widehat{\x \vecu \vecv} > \pi/2$. Then $\|\vecu - \vecv\| \leq \|\x - \vecv\|$, which contradicts $\x = p_\A(\vecv)$. 
Hence $\vecv \in \Emb$ and $\x$ has a pre-image by $p_\A$ in $\Emb$.

Now, suppose $\exists \x \in p_\A(\X)$ such that its pre-image(s) in $\X$ by $p_\A$ belong to $\X \setminus \Emb$.
Denote $\x' \in p_\A(\X)$ the closest vertex of $p_\A(\X)$, which has a pre-image in $\Emb$ by $p_\A$.
By continuity of $p_\A$, there exists $\x'' \in [\x, \x']$ with pre-image in $(\X \setminus \Emb) \cap \Emb = \emptyset$, hence there is a contradiction since $\x''$ has at least one pre-image. Consequently $\x$ has at least a pre-image in $\Emb$, and $p_\A(\X) \subseteq p_\A(\Emb)$.
\end{proof}

\subsection{Proof of Theorem \ref{th:2}}
\label{ap:proof3}
\begin{proof}

As a preliminary, let us show that $\forall \y \in \Z$, $\gamma(\y) \in \Emb$.
Let $\x_1 \in \X \cap p_\A^{-1}(\B^\top \y) (\neq \emptyset)$.
From Proposition \ref{prop:eq_zon}, $\y$ also have a pre-image $\x_2 \in \Emb$ by $p_\A$,
and denote $\vecu \in \Ran(\A)$ such that $p_\X(\vecu) = \x_2$, i.e.,
$\| \x_2 - \vecu \| = \min \limits_{\x \in \X} \| \x - \vecu \|$. 
We have $\| \x_1 - \vecu \|^2 = \| \x_1 - \B^\top \y \|^2 +  \| \B^\top \y - \vecu \|^2$ and $\| \x_2 - \vecu \|^2 = \| \x_2 - \B^\top \y \|^2 +  \| \B^\top \y - \vecu \|^2$ as $\x_1, \x_2 \in p_\A^{-1}(\B^\top \y)$.
Then, $\| \x_2 - \vecu \| \leq \| \x_1 - \vecu \| \Rightarrow \| \x_2 - \B^\top \y \| \leq \| \x_1 - \B^\top \y \|$ with equality if $\x_1 = \x_2$, so that $\gamma(\B^\top \y) \in \Emb$.

We now proceed by showing that $\gamma$ defines a bijection from $\Z$ to $\Emb$, with $\gamma^{-1} = \B$.
First, $\forall \y \in \Z$, $\B \gamma (\y) = \y$ by definition of $\gamma$. 
It remains to show that, $\forall \x \in \Emb$, $\gamma(\B \x) = \x$.
Let $\x \in \Emb$, $\vecu \in \Ran(\A)$ such that $p_\X(\vecu) = \x$. 
Suppose $\gamma(\B \x) = \x' \in \Emb$, $\x' \neq \x$, in particular $\| \x' - \B^\top \B \x \| < \| \x - \B^\top \B \x \|$. Again, $\x, \x' \in p_\A^{-1}(\B^\top \B \x)$, hence $\| \x' - \B^\top \B \x \|^2 + \|  \B^\top \B \x - \vecu \|^2 = \| \x' - \vecu\|^2  < \| \x - \B^\top \B \x \|^2 + \|  \B^\top \B \x - \vecu\|^2 = \| \x - \vecu\|^2$ which contradicts $\x = p_\X(\vecu)$. Thus $\gamma(\B \x) = \x$.

$\Z$ is compact, convex and centrally symmetric from being a zonotope, see Definition \ref{def:zon}. 
Finally, any smaller set than $\Z$ would not have an image through $\gamma$ covering $\Emb$ entirely, which concludes the proof. 
\end{proof}

\subsection{Proof of Proposition \ref{prop:vols}}
\label{ap:proof4}
\begin{proof}
The first part directly follows from the properties of the convex and orthogonal projection. 
In detail: $\Vol_d(\A \U) \geq \Vol_d(p_\X(\A \U)) = \Vol_d(\Emb) \geq \Vol_d(p_\A(\Emb)) = \Vol_d(\A\Z)$.

For the second part, we need the length of a strip $\Band_i$, i.e., the inter hyperplane distance: $l_i = 2 / \| \A_i \|$. Recall that $\B = \A^\top$, that rows of $\A$ have equal norm and $\A$ orthonormal. Then, following the proof of \cite[Theorem 1]{Filliman1988}, $\sum \limits_{j = 1}^d \| \B_j \|^2 = d$ (orthonormality) $= \sum \limits_{i = 1}^D \sum \limits_{j = 1}^d A_{i,j}^2 = \sum \limits_{i = 1}^D \| \A_i \|^2 = D  \| \A_1 \|^2$, hence $\| \A_1 \| = \sqrt{d/D}$.  As $\Z$ is enclosed in the $d$-sphere of radius $\sqrt{D}$ and the $d$-sphere of radius $\sqrt{D/d}$ is enclosed in $\I$, the result follows from the formula of the volume of a $d$-sphere of radius $\rho$: $\frac{\pi^{d/2}}{\Gamma(d/2 + 1)} \rho^d$ .    
\end{proof}

\section{Main notations}

\begin{table}[h]
\begin{tabular}{ll}
 $d$ & low embedding dimension\\
 $D$ & original dimension, $D \gg d$\\
 $\A$ & random embedding matrix of size $D \times d$\\
 $\B$ & transpose of $\A$ after orthonormalization\\
 $\X$ & search domain $[-1,1]^D$ \\
 $\Y$ & low dimensional optimization domain, in $\R^d$\\
 $\Z$ & zonotope $\B \X$\\ 
 $p_\X$ & convex projection onto $\X$\\
 $p_\A$ & orthogonal projection onto $\Ran(\A)$\\
 $\Psi$ & warping function from $\R^d$ to $\Ran(\A)$\\
 $\phi$ & mapping from $\Y \subset \R^d$ to $\R^D$\\
 $\gamma$ & mapping from $\Z \subset \R^d$ to $\R^D$\\
 $(\pbR)$ & optimization problem for REMBO with mapping $\phi$\\
 $(\pbR')$ & optimization problem for REMBO with mapping $\gamma$\\
 $(\pbQ)$ & minimal volume problem definition with mapping $\phi$\\
 $(\pbQ')$ & minimal volume problem definition with mapping $\gamma$\\ 
 $\boxx$ & box enclosing $\Z$\\
 $\Emb$ & image of $\R^d$ by $\phi$\\
 $\Band_i$ & strip associated with the $i^{th}$ row of $\A$\\
 $\I$ & intersection of all strips $\Band_i$\\
 $\U$ & union of all intersection of $d$ strips $\Band_i$\\
 $\Paral_I$ & parallelotope associated with strips in the set $I$\\

\end{tabular}
\end{table}

\end{document}